 \documentclass[preprint,12pt]{elsarticle}
\usepackage{amssymb,amsmath,amsthm,bm,stmaryrd,graphicx,subfigure,float,graphicx,nameref}
\usepackage[colorlinks,pdfborder=001,linkcolor=blue,anchorcolor=blue,citecolor=blue]{hyperref}
\usepackage{epstopdf}
\usepackage{tikz}
\usepackage[all]{xy}
\usepackage{amsthm}
\usepackage{graphicx}
\usepackage{lineno}
\usepackage{CJK}

\usetikzlibrary{positioning} 
\usepackage{graphicx}
\usepackage{color}
\newtheorem{thm}{Theorem}[section]
\newtheorem{lem}[thm]{Lemma}
\newtheorem{ex}[thm]{Example}
\newtheorem{cor}[thm]{Corollary}
\newtheorem{main-thm}{Main Theorem}

\newtheorem{prob}[thm]{Problem}

\theoremstyle{definition}
\newtheorem{definition}{Definition}

\theoremstyle{remark}
\newtheorem{rem}{Remark}

\newcommand{\R}{{\mathbb R}}                    
\newcommand{\N}{{\mathbb N}}                    
\newcommand{\II}{{\mathbb I}}                    
\newcommand{\sphere}{{\mathbf S}}               

\newcommand{\Q}{{\mathbb Q}}
\newcommand{\PP}{{\mathbb P}}
\newcommand{\0}{{\emptyset}}

\newcommand{\cl}{\operatorname{cl}}             
\newcommand{\cf}{\operatorname{cf}}             

\newcommand{\Cld}{\operatorname{Cld}}           
\newcommand{\sm}{\setminus}

\newcommand{\Clopen}{\operatorname{\Clopen}}           

\begin{document}
\begin{frontmatter}

\title{ Axioms of Continuous Separation \tnoteref{t1}}

\tnotetext[t1]{The work was supported by National Natural Science Foundation of China (No. 12571077).}
\author[Z. Yang]{Zhongqiang Yang}
 \address[Z. Yang]{School of Mathematical Sciences, Center for Applied Mathematics of Guangxi, Guangxi Minzu
University, Nanning 530006, China}
\ead{zqyang@stu.edu.cn}

\author[N. Wu]{Nada Wu*}

\cortext[mycorrespondingauthor]{Corresponding author.}

\address[N. Wu]{School of Mathematics and Statistics,
	Hanshan Normal University,
	Chaozhou, Guangdong 521041, China}

\ead{ndwu@hstc.edu.cn}

\begin{abstract}
\par For a $T_1$-space $X$, let $\Cld(X)$ be the class of all non-empty  closed sets in $X$ and $\mathcal{T}_4(X)=\{(F_1,F_2)\in \Cld(X)^2\mid F_1\cap F_2=\emptyset\}$. It is well known that $X$ is  $T_4$ if and only if for each $(F_1,F_2)\in \mathcal{T}_4(X)$, there exists a pair of closed sets $\phi_1(F_1,F_2)$ and $\phi_2(F_1,F_2)$ such that $\phi_1(F_1,F_2)\cup \phi_2(F_1,F_2)=X$ and $F_j\cap \phi_j(F_1,F_2)=\emptyset$ for $j=1,2$.
Motivated by this characterization, for a topology $\tau$ on $\Cld(X)$, we call $X$ $CT_4$ with respect to $\tau$ if such maps $\phi_j : \mathcal{T}_4(X) \to \Cld(X)$ can be chosen to be continuous. Analogously, we define $CT_i$-spaces for $i = 1,2,3$. In this paper, $\Cld(X)$ is always endowed with the Vietoris topology. We prove that every $CT_4$-space is countably compact, every  $CT_3$-space is a  Fr\'{e}chet-Urysohn space, and every separable subspace of a $CT_3$-space is metrizable. Many examples are provided. All finite-dimensional cubes, the infinite-dimensional cube, all finite-dimensional spheres, and every zero-dimensional compact metrizable space are $CT_4$. All infinite discrete spaces, all finite-dimensional Euclidean spaces, and all countable limit ordinal spaces are $CT_3$ but not $CT_4$. In addition, every metrizable space with a unique non-isolated point is $CT_3$, and it is $CT_4$ if and only if it is compact. Moreover, the topological sum of infinitely many $CT_3$-spaces is $CT_3$ but not $CT_4$. Every countable space with a unique non-isolated point is $CT_2$, and such a space is $CT_3$ precisely when it is metrizable. Every  proper subfield of the real numbers field $\mathbb{R}$ and its complement in $\mathbb{R}$, equipped with the usual Euclidean topology, are $CT_2$ but not $CT_4$; whether these spaces are $CT_3$ remains open. All uncountable ordinal spaces and the one-point compactification of any uncountable discrete space are non-$CT_2$.  Finally, $CT_1$ and $T_1$ are equivalent.  We propose several open problems. In particular, does there exist a non-metrizable $CT_3$-space or a non-metrizable $CT_4$-space? Is every compact $CT_2$-space necessarily $CT_3$ or $CT_4$?
\end{abstract}

\end{frontmatter}

\section{Definitions}

Most of our notation and terminology follows Engelking \cite{Engelking-1989}.
For example, $|A|$ is the cardinal of a set $A$, and
\[  B(x,\delta)=\{y\in X\mid d(x,y)<\delta\}\]
 for a metric space $(X,d)$, $x\in X$ and $\delta>0$. Moreover, $w(X)$ and $d(X)$ are the weight and density of a topological space $X$, respectively. In   what follows, we list several differences from the book. For a map $f\colon X\to Y_1\times Y_2$, we use $f_1$ and $f_2$ to denote the compositions of $f$ with the projections from $Y_1\times Y_2$ onto $Y_1$ and $Y_2$, respectively. We use $A\subset B$ or $B\supset A$ to express that $A$ is a proper subset of $B$, and $A\subseteq B$ or $B\supseteq A$ to express that $A$ is a subset of $B$. We use $\cl A$ to denote the closure of $A$ in a topological space.
 A clopen set in a space means that it is a closed and open set in the space. {\bf We assume that all topological spaces are $T_1$ and non-empty.}

 For a space $X$, let $\Cld(X)$ be the class of all non-empty  closed sets in $X$.
And let
\[
\begin{split}
 \mathcal{T}_4(X)=\{(F_1,F_2)\in \Cld(X)^2\mid F_1\cap F_2=\emptyset\};\\
\mathcal{T}_3(X)=\{(F_1,F_2)\in \mathcal{T}_4(X)\mid F_1\mbox { is a singleton}\};\\
\mathcal{T}_2(X)=\{(F_1,F_2)\in \mathcal{T}_4(X)\mid F_1\mbox { and } F_2\mbox { are singletons}\}.
\end{split}
\]
Recall that for $i=2,3,4$, a space $X$ is called $T_i$ if for each $(F_1,F_2)\in \mathcal{T}_i$, there exists a pair of disjoint open sets $U_1,U_2$  such that $F_j\subseteq U_j$ for $j=1,2$. Instead of using open sets, we may reformulate these concepts in terms of closed sets.  A space $X$ is called $T_i$ if for each $(F_1,F_2)\in \mathcal{T}_i$, there exists a pair of  closed sets $\phi_1(F_1,F_2)$ and $\phi_2(F_1,F_2)$ such that $\phi_1(F_1,F_2)\cup \phi_2(F_1,F_2)=X$ and $F_j\cap \phi_j(F_1,F_2)=\emptyset$ for $j=1,2$.

Hence, for a natural topology $\tau$ on $\Cld(X)$, we can define the following concepts:
\begin{definition} For $i=2,3,4$, a topological space $X$ is called {\bf continuously $T_i$ with respect to the topology $\tau$} ($CT_i$,  in short), if there exists a continuous map $\phi:\mathcal{T}_i(X)\to \Cld(X)^2$, where $\mathcal{T}_i(X)$ is endowed with the subspace topology of the product space $(\Cld(X),\tau)^2$, such that for each $(F_1,F_2)\in \mathcal{T}_i(X)$, $\phi_1(F_1,F_2)\cup \phi_2(F_1,F_2)=X$ and
$F_j\cap \phi_j(F_1,F_2)=\emptyset$ for $j=1,2$. Then, we call the map $\phi$ a  {\bf continuous $T_i$ map for the space $X$
with respect to the topology $\tau$} ( a {\bf $CT_i$-map}, in short).
\end{definition}
 We can also define:

\begin{definition} A topological space $X$ is called {\bf continuously $T_1$} {\bf with respect to the topology} $\tau$  ($CT_1$, in short), if  there exists a continuous map $\phi:X^2\setminus \triangle(X)\to \Cld(X)^2$, where
\[  \triangle(X)=\{(x,x)|x\in X\}\]
 is the {\bf diagonal of}  $X$, such that
$x_1\not\in \phi_1(x_1,x_2)$, $x_2\not\in \phi_2(x_1,x_2)$ and $x_1\in\phi_{2}(x_1,x_2)$, $x_2\in\phi_{1}(x_1,x_2)$  for each $(x_1,x_2)\in X^2\setminus \triangle(X)$.  Then, we call the map $\phi$  a  {\bf continuous $T_1$ map for the space $X$ with respect to the topology $\tau$} ( a {\bf  $CT_1$-map}, in short).
\end{definition}

For a topological space $X$ and a set $U$ in $X$, let
\begin{gather*}
U^- = \{ F \in \Cld(X) \mid F \cap U \neq \emptyset \},\\
U^+ = \{ F \in \Cld(X) \mid F \subseteq U\}.
\end{gather*}
The topology on $\Cld(X)$ which is generated by $\{U^-,U^+\mid U$ is open in $X\}$ as a subbase is called the {\bf Vietoris topology} on $\Cld(X)$.
Then, $\{\langle U_1,U_2,\cdots, U_n\rangle\mid U_i\mbox{ is open in }X \mbox { for } i\leq n, n\in\N\},$ where
\[  \langle U_1,U_2,\cdots, U_n\rangle=\bigcap_{i=1}^n U_i^-\cap (\bigcup_{i=1}^n U_i)^+,\]
is a base for the  Vietoris topology on $\Cld(X)$.

Let $(X,d)$ be a metric space. For any $A\in\Cld(X),$ $x\in X$ and $\varepsilon>0$, let
\[
d(x,A) = \inf_{a\in A} d(x,a), \quad
B(A,\varepsilon)=\{x\in X\mid d(x,A)<\varepsilon\}.
\]
Recall that the \textbf{Hausdorff distance} $d_H$ on $\Cld(X)$ is defined as follows:
For any $A,B\in\Cld(X)$, the Hausdorff distance between $A$ and $B$ is
\[
d_H(A,B) =\inf\{\varepsilon>0\mid B\subseteq B(A,\varepsilon)
~{\rm and}~ A\subseteq B(B,\varepsilon)\}\leq +\infty.
\]
It is well known that if $(X,d)$ is compact, the Hausdorff distance is indeed a metric on $\Cld(X)$ and induces the Vietoris topology on $\Cld(X)$.

In what follows, for a topological space $X$, when discussing the properties of  $CT_i$ ( for $i=1,2,3,4$),
we always endow $\Cld(X)$ with the Vietoris topology.

\begin{rem} Originally, by the above definitions,
$\mathcal{T}_2(X) = \big\{(\{a\},\{b\}) \mid a,b\in X \text{ and } a\neq b\big\}$ is a subspace of $\Cld(X)^2$.
However, for notational convenience in the proofs below involving $\mathcal{T}_2(X)$,
we generally write $\mathcal{T}_2(X)$ as $\{(a,b)\in X^2 \mid a\neq b\}=X^2\setminus \triangle(X)$ and regard it as a subspace of the product space $X^2$.
This is acceptable because these two spaces are obviously homeomorphic. In some of the proofs below,
$\mathcal{T}_3(X)$ will sometimes be used in a similar manner, and we shall use this identification without further comment.
\end{rem}

\section{Basic properties}

Trivially, every $CT_i$-space is $T_i$ for $i=1,2,3,4$, and every $CT_i$-space is $CT_{i-1}$ for $i=2,3,4$. Moreover, we have the following
simple theorem:

\begin{thm}\label{T_1-is-CT1} A topological space is $T_1$ if and only if it is  $CT_1$.
\end{thm}
\begin{proof}It suffices to prove that if $X$ is $T_1$, then $X$ is  $CT_1$.
	We define the natural map $	\phi: X^2 \setminus \triangle(X) \to \Cld(X)^2$ by,  for $(x,y) \in X^2 \setminus \triangle(X)$
	\[
	\phi(x,y) = (\{y\}, \{x\}).
\]
	Since $X$ is $T_1$, $\phi_1(x,y)=\{y\}$ and $\phi_2(x,y)=\{x\}$ are closed in  $X$. Moreover, it is trivial that $x \notin \phi_1(x,y)$, $x \in \phi_2(x,y)$,
	$y \notin \phi_2(x,y)$ and $y \in \phi_1(x,y)$. It is not hard to verify that $\phi$ is continuous.
Thus, $\phi$ is a $CT_1$-map for $X$ and hence $X$ is  $CT_1$.
\end{proof}

While the classical separation axioms possess rich permanence properties, the continuous separation axioms proposed in this paper are currently known to satisfy only the following permanence properties (Theorems \ref{clopen-subspace} and \ref{sum}). Define $\Cld^*(X)=\Cld(X)\oplus\{\emptyset\}$.

\begin{lem}\label{cap-cup}
For each space $X$ and every clopen subspace $Y$ of $X$, the map $f:\Cld(X)\to \Cld^*(Y)$ defined by $A\mapsto A\cap Y$ is continuous, and the union map $\cup:\Cld(X)^2\to\Cld(X)$ is also continuous.
\end{lem}
\begin{proof}
For any open set $V \subseteq Y$, since $Y$ is a clopen subset of $X$, $V$ and  $V\cup(X\sm Y)$ are also open in $X$.
Let $V_Y^-=\{ B \in \Cld^*(Y)  \mid B \cap V \neq \emptyset \}, \quad
V_Y^+ = \{ B \in \Cld^*(Y)  \mid B \subseteq V\}.
$
Then
\[  	f^{-1}(V_Y^-) = \{ A \in \Cld(X) \mid (A \cap Y) \cap V\neq \emptyset \}
= \{ A \in \Cld(X) \mid A \cap V \neq \emptyset \}
= V^-,
\]
\[ 	f^{-1}(V_Y^+) = \{ A \in \Cld(X) \mid A \cap Y\subseteq V\}
	= \{ A \in \Cld(X) \mid A \subseteq V\cup (X\setminus Y)\}
	=\bigl(V\cup (X\setminus Y)\bigr)^+.\]
Hence, $f$ is continuous.
Moreover,
	\[
	\quad \cup^{-1}(V^-) = \big(V^- \times \operatorname{Cld}(X)\big) \cup \big(\operatorname{Cld}(X) \times V^-\big) \text{ and } \cup^{-1}(V^+) = V^+ \times V^+ ,
	\]
	both of them are open in $\Cld(X) \times \Cld(X)$. Hence, the union map $\cup$ is continuous.
\end{proof}

\begin{lem}\label{extension}
Let $X$ be a $CT_3$-space and $\phi:\mathcal{T}_3(X)\to \Cld(X)^2$ be a  $CT_3$-map.
 Define
\[  \widetilde{\mathcal{T}_3}(X)=\mathcal{T}_3(X)\oplus  (X\times\{\emptyset\}).\]
An extension $\widetilde{\phi}:\widetilde{\mathcal{T}_3}(X)\to \Cld^*(X)^2$
of $\phi$  is defined as follows: for any $x\in X$,
 \[   \widetilde{\phi}(x,\emptyset)=(\emptyset,X).\]
Then  $\widetilde{\phi}$ is continuous.
\end{lem}

\begin{proof}
The restriction of $\widetilde{\phi}$  to $X\times\{\emptyset\}$
is a constant map and hence, it is continuous. Thus,  $\widetilde{\phi}$ is continuous.
\end {proof}

\begin{lem}\label{extension-4}
Let $X$ be a $CT_4$-space and $\phi:\mathcal{T}_4(X)\to \Cld(X)^2$ be a  $CT_4$-map.
 Define
\[  \widetilde{\mathcal{T}_4}(X)=\mathcal{T}_4(X)\oplus (\{\emptyset\}\times \Cld(X))\oplus (\Cld(X)\times\{\emptyset\})\oplus \{(\emptyset,\emptyset)\}.\]
An extension $\widetilde{\phi}:\widetilde{\mathcal{T}_4}(X)\to \Cld^*(X)^2$ of $\phi$, is defined as follows: for any $A,B\in \Cld(X)$,
 \[  \widetilde{\phi}(\emptyset,B)=(X,\emptyset),~~ \widetilde{\phi}(A,\emptyset)=(\emptyset,X),~~ \widetilde{\phi}(\emptyset,\emptyset)=(X,X).\]
Then, $\widetilde{\phi}$ is continuous.
\end{lem}

\begin{proof}
The proof is analogous to that of Lemma \ref{extension}.
\end{proof}
\begin{thm}\label{clopen-subspace}
For $i=1,2,3,4$, let $X$ be a $CT_i$-space and $Y$ a clopen subspace of $X$. Then $Y$ is also $CT_i$.
\end{thm}
\begin{proof} Let $f:\Cld(X)\to \Cld^*(Y)$ be the map defined in Lemma \ref{cap-cup}.
For $i=2,3,4$, if $\phi:\mathcal{T}_i(X)\to \Cld(X)^2$ is a $CT_i$-map for $X$, then
 it follows from Lemma \ref{cap-cup} that the composite map $(f\times f)\circ\phi|_{\mathcal{T}_i(Y)}:\mathcal{T}_i(Y)\to \Cld^*(Y)^2$ is a $CT_i$-map for $Y$, since the image of the composite map is included in $\Cld(Y)^2$. For $i=1$, the statement follows immediately from Theorem \ref{T_1-is-CT1}.
 \end{proof}

\begin{thm}\label{sum}
Let $\{X_s\mid s\in S\}$ be a family of spaces. For $i=1,2,3$, $X=\bigoplus_{s\in S} X_s$ is  $CT_i$ if and only if each $X_s$ is $CT_i$.
Moreover, $X=\bigoplus_{s\in S} X_s$ is $CT_4$ if and only if $S$ is finite and each $X_s$ is $CT_4$.
\end{thm}

\begin{proof}
Suppose $i=3$. The ``only if'' part follows from Theorem \ref{clopen-subspace}. Now, we show the ``if'' part.
For each $s\in S$, let $\phi^s:\mathcal{T}_3(X_s)\to\Cld(X_s)^2$ be a $CT_3$-map for the space $X_s$. By Lemma \ref{extension}, we can continuously
 extend them to $\widetilde{\phi^s}:\widetilde{\mathcal{T}_3}(X_s)\to \Cld^*(X_s)^2$. Moreover, we define $\phi:\mathcal{T}_3(X)\to\Cld(X)^2$ as
 follows: For $(x,B)\in\mathcal{T}_3(X)$, there exists a unique  $s_0\in S$ such that $x\in X_{s_0}$, let
\[
\phi(x,B)=\big(\widetilde{\phi_1^{s_0}}(x, B\cap X_{s_0})\cup\bigcup_{s\in S\setminus\{s_0\}} X_s,~\widetilde{\phi_2^{s_0}}(x, B\cap X_{s_0})\big).
\]
 It follows from Lemmas \ref{cap-cup} and \ref{extension} that $\phi:\mathcal{T}_3(X)\to\Cld(X)^2$ is continuous. Trivially, it also satisfies the other conditions. Hence, $X$ is $CT_3$.

The case $i=2$ is analogous.

The case $i=1$ follows from Theorem \ref{T_1-is-CT1}.

Now, we consider the case $i=4$.

To show the ``if'' part, it suffices to verify that $X_1\oplus X_2$ is $CT_4$ if $X_1$ and $X_2$ are $CT_4$.
Let $\phi^1:\mathcal{T}_4(X_1)\to\Cld(X_1)^2$ and $\phi^2:\mathcal{T}_4(X_2)\to\Cld(X_2)^2$
be  $CT_4$-maps for $X_1$ and $X_2$, respectively. By Lemma \ref{extension-4}, we can define continuous extensions
$\widetilde{\phi^1}:\widetilde{\mathcal{T}_4}(X_1)\to \Cld^*(X_1)^2$ of $\phi^1$ and $\widetilde{\phi^2}:\widetilde{\mathcal{T}_4}(X_2)\to \Cld^*(X_2)^2$  of $\phi^2$.
Now, we define $\phi:\mathcal{T}_4(X)\to \Cld(X)^2$ as follows:
\[
\begin{split}
\phi_1(A,B)=\widetilde{\phi_1^1}(A\cap X_1, B\cap X_1)\cup \widetilde{\phi_1^2}(A\cap X_2, B\cap X_2),\\
\phi_2(A,B)=\widetilde{\phi_2^1}(A\cap X_1, B\cap X_1)\cup \widetilde{\phi_2^2}(A\cap X_2, B\cap X_2).
\end{split}
\]
By Lemma \ref{cap-cup}, $\phi:\mathcal{T}_4(X)\to \Cld(X)^2$ is continuous. Clearly, it also satisfies the other conditions in the definition of $CT_4$.
Thus, $X$ is $CT_4$.

For the ``only if'' part, if $X$ is $CT_4$, using Theorem \ref{clopen-subspace}, each $X_s$ is $CT_4$. Note that any
infinite topological sum is not countably compact. It follows from Theorem \ref{non-T4} that $S$ is finite.
\end{proof}

\section{Examples of  $CT_4$-spaces}

Trivially, each finite space is  a  $CT_4$-space.
We now give more examples.

\begin{definition}
	Let $(X,d)$ be a metric space. We say that $(X,d)$ has the {\bf closure property} ({\bf CP}, in short) if
	for every pair $(A,B) \in \mathcal{T}_4(X)$,
	\[
	\operatorname{cl}\{ x \mid d(x,A) > d(x,B) \}
	= \{ x \mid d(x,A) \ge d(x,B) \}.
	\]
\end{definition}
Note that the CP depends on the choice of the metric $d$. For example, with the usual metric, $\II=[0,1]$ has the CP. However,
consider the space
\[  X=\overline{(-1,0)(0,1)}\cup \overline{(0,1)(0,-1)}\cup\overline{(0,-1)(1,0)},\]
 where $\overline{AB}$ is the segment with the endpoints $A$ and $B$, and
 	$X$ is equipped with the Euclidean metric of $\R^2$. Then $X$ is homeomorphic to $\II$ without the CP.

\begin{thm}\label{compact metric cp}Let $(X, d)$ be a compact metric space with the CP. Then $(X, d)$ is  $CT_4$.
\end{thm}

\begin{proof}
For $(A,B) \in \mathcal{T}_4(X)$, and $m \in \mathbb{N}$, where $\mathbb{N}$ is the set of natural numbers, let
\[
\begin{split}
E(A,B) = \{ x \in X \mid d(x,A) \ge d(x,B) \}, \\
E_m(A,B) = \{ x \in X \mid d(x,A) \ge d(x,B) + \tfrac{1}{m} \}.
\end{split}
\]
It is trivial to verify that $E(A,B)$ and  $E_m(A,B)$ are closed, and $E_m(A,B)\subseteq E(A,B)$.

Define $\phi: \mathcal{T}_4(X) \to \Cld(X)^2$ by
\[  \phi(A,B) =\left(E(A,B),E(B,A)\right), \mbox{ for  } (A,B) \in \mathcal{T}_4(X).\]
Clearly, $\phi_1(A,B)\cup \phi_2(A,B)=X$,
$A\cap \phi_1(A,B)=\emptyset$, and $B\cap \phi_2(A,B)=\emptyset.$
Hence, it remains only to prove that $\phi$ is continuous.

\medskip

\noindent
\textbf{Claim 1.} \textit {For all $(A,B), (A',B') \in \mathcal{T}_4(X)$, if $d_H(A',A) < \tfrac{1}{2m}$ and $d_H(B',B) < \tfrac{1}{2m}$, then
$E_m(A',B') \subseteq E(A,B)$.}

For any $x \in E_m(A',B')$, choose $a \in A$, $a' \in A'$ such that
\[  d(x,A) = d(x,a) \quad \text{and} \quad d(a,a') < \tfrac{1}{2m}.\]
Then
\begin{equation}\label{equ-1}
\begin{split}
d(x,A)=d(x,a) &\geq d(x,a')- d(a,a') \\
	&>d(x,A')- \tfrac{1}{2m}\\
	&\geq d(x,B')+\tfrac{1}{m}-\tfrac{1}{2m}.
\end{split}
\end{equation}
Choose $b' \in B'$ and $b \in B$ such that
\[
d(x,B') = d(x,b') \quad \text{and} \quad d(b,b') < \tfrac{1}{2m}.
\]
Then
\begin{equation*}
d(x,B') = d(x,b') \ge d(x,b) - d(b,b') > d(x,B) - \tfrac{1}{2m}.
\end{equation*}
It follows from Equation (\ref{equ-1}) that
\[
d(x,A)>d(x,B) - \tfrac{1}{2m} + \tfrac{1}{m} - \tfrac{1}{2m} = d(x,B).
\]
Thus $x \in E(A,B)$, that is, $E_m(A',B') \subseteq E(A,B)$. We have completed the proof of Claim 1.
\\

In the following, we show that $d_H(E(A_n,B_n),E(A,B))\to 0$ under the assumption that $(A_n,B_n)\to(A,B)$ (i.e., $d_H(A_n,A)\to 0$ and $d_H(B_n,B)\to 0$).

\medskip

\noindent
\textbf{Claim 2. (Pointwise convergence)} For every $x\in E(A,B)$, $d(x,E(A_n,B_n))\to 0.$

If $d(x,A)>d(x,B)$, choose $m$ such that $x\in E_m(A,B)$. For all sufficiently large $n$, we have $d_H(A_n,A)<\tfrac{1}{2m}$ and $d_H(B_n,B)<\tfrac{1}{2m}$; by Claim 1, this implies $E_m(A,B)\subseteq E(A_n,B_n)$, so,  $x\in E(A_n,B_n)$. It follows that  $d(x,E(A_n,B_n))=0$  for sufficiently large $n$.

If $d(x,A)=d(x,B)$, by CP we have
\[
x\in \operatorname{cl}\{z:d(z,A)>d(z,B)\}.
\]
Thus there exists a sequence $(x_k)$ such that $x_k\to x$ and $d(x_k,A)>d(x_k,B)$ for each $k$.

Now we show that $d(x,E(A_n,B_n))\to 0$. For arbitrary $\varepsilon>0$, choose $k_0$ such that $d(x,x_{k_0})<\varepsilon$. By the argument in the previous paragraph,
there exists $N$ such that $x_{k_0}\in E(A_n,B_n)$ when $n>N$. It follows that 
\[
d(x,E(A_n,B_n))\le d(x,x_{k_0})<\varepsilon,~~\mbox{if}~~n>N.
\]
Thus Claim 2 holds.

\medskip

\noindent
\textbf{Claim 3. (Uniform convergence)} $\sup_{x\in E(A,B)} d(x,E(A_n,B_n))\to0.$

Suppose not. Then there exist $\varepsilon>0$, a subsequence $(n_j)$, and a sequence of points $(x_j)$ in $E(A,B)$ such that
\[
d(x_j,E(A_{n_j},B_{n_j}))\ge \varepsilon
\]
for all $j$.
Since $X$ is compact and $E(A,B)$ is closed in $ X$, the sequence $(x_j)$ has a convergent subsequence; passing to a subsequence and relabeling, we may assume $x_j\to x$ for some $x\in E(A,B)$. Then for sufficiently large $j$,
\[
d(x,E(A_{n_j},B_{n_j}))\ge d(x_j,E(A_{n_j},B_{n_j}))-d(x_j,x)\ge \frac{\varepsilon}{2}.
\]
This contradicts Claim 2, which says $d(x,E(A_n,B_n))\to0$. Thus Claim 3 holds.

\medskip

\noindent
\textbf{Claim 4.} If $x_n\in E(A_n,B_n)$ and $x_n\to x$, then $x\in E(A,B)$.

Since $x_n\in E(A_n,B_n)$, we have
\[
d(x_n,A_n)\ge d(x_n,B_n).
\]
Because $d_H(A_n,A)\to0$ and $d_H(B_n,B)\to0$, the distance functions $d(\cdot,A_n):X\to\mathbb R$ and $d(\cdot,B_n):X\to\mathbb R$ converge uniformly on $X$ to $d(\cdot,A):X\to\mathbb R$ and $d(\cdot,B):X\to\mathbb R$, respectively. Hence $d(x_n,A_n)\to d(x,A), d(x_n,B_n)\to d(x,B).$ Therefore, $d(x,A)\ge d(x,B)$, that is, $x\in E(A,B)$. We have completed the proof of Claim 4.

Now, for every $\varepsilon>0$, by Claim 4 and the compactness of $X$, there exists $N_1>0$ such that for all $n\ge N_1$,
\[
E(A_n,B_n)\subseteq N_\varepsilon(E(A,B))=\{x\in X\mid d(x,E(A,B))<\varepsilon\}.
\]
Indeed, if not, there are a subsequence $(n_j)$ and points $x_j\in E(A_{n_j},B_{n_j})$ with $d(x_j,E(A,B))\ge\varepsilon$. By compactness, $(x_j)$ has a subsequence $(x_{j_k})$ which converges to $x$. It follows from Claim 4 that $x\in E(A,B)$. So, $d(x_j,E(A,B))\le d(x_j,x)\to 0$, a contradiction occurs.

On the other hand, from Claim 3 it follows that there exists $N_2>0$ such that for all $n\ge N_2$,
\[
E(A,B)\subseteq N_\varepsilon(E(A_n,B_n))=\{x\in X\mid d(x,E(A_n,B_n))<\varepsilon\}.
\]

Hence for all $n\ge \max\{N_1,N_2\}$, $E(A_n,B_n)\subseteq N_\varepsilon(E(A,B))$, $E(A,B)\subseteq N_\varepsilon(E(A_n,B_n))$.
That is, $d_H(E(A_n,B_n),E(A,B))\to0.$

Thus $\phi_1$ is continuous. Similarly, $\phi_2$ is also continuous.
\end{proof}

\begin{thm}\label{compact convex cp}Let $X$ be a compact convex set in a  Hilbert space $(L, \langle\cdot,\cdot \rangle)$. Then $(X,d)$ has the CP, where $d$ is
the metric on $X$ induced by the inner product.
\end{thm}

\begin{proof}
For any $(A,B)\in \mathcal{T}_4(X)$, let
\[
E^0(A,B)=\{x\in X \mid d(x,A) > d(x,B)\}, \quad
E(A,B)=\{x\in X \mid d(x,A) \ge d(x,B)\}.
\]
We aim to prove that $\operatorname{cl}E^0(A,B)=E(A,B).$

It is easily seen that $E(A,B)$ is closed,
which immediately implies
$\operatorname{cl}E^0(A,B)\subseteq E(A,B).$
Therefore, it remains only to establish the reverse inclusion
$\operatorname{cl}E^0(A,B)\supseteq E(A,B).$

Let
\[F(A,B)=\{x\in X \mid d(x,A)=d(x,B)\}.\]
Note that 	$E(A,B)=E^0(A,B)\cup F(A,B)$,
it suffices to prove that $F(A,B)\subseteq\operatorname{cl}E^0(A,B)$.

    For any $x\in F(A,B)$, let $r=d(x,A)=d(x,B)>0$, then there exists  $b \in  B$ such  that $\|x-b\|=r$.
		 For any 	$\varepsilon \in (0,r)$, let
	\[
	x' = \tfrac{r-\varepsilon}{r}x + \tfrac{\varepsilon}{r}b,
	\]
	then $\|x' - x\|=\dfrac{\varepsilon}{r} \|x - b\|=\varepsilon.$
In what follows, we shall prove that $x'\in E^0(A,B).$
First,
		\[  \|x' - b\|= \left\| \tfrac{r-\varepsilon}{r}x + \tfrac{\varepsilon}{r}b - b \right\|
		= \tfrac{r-\varepsilon}{r} \|x - b\| = r - \varepsilon.
		\]
Second, for any $a \in A$,
	\[
	\|x' - a\| \geq \|x - a\| - \|x' - x\| \geq r - \varepsilon.
	\]
	If $\|x' - a\| =  r - \varepsilon,$ then $\|x- a\|=r$ and $\|x' - a\| = \|x - a\| - \|x' - x\|$. Thus $a, x', x, b$ are collinear and $x'$ lies between $x$ and $a$,
	which contradicts the fact that $x'$ lies between $x$ and $b$, $\|x'-a\|=\|x'-b\|$ and $a\not=b$.
		 Hence  $\|x' - a\|>r - \varepsilon$. It follows that $	d(x',A)>r - \varepsilon$.
		Therefore,	$	d(x',A) > d(x',B).	$ We are done.
\end{proof}

\begin{ex}\label{cube} For each $n\in\N\cup\{\infty\}$, $[0,1]^n$ is  $CT_4$.\end{ex}
\begin{proof} For $n\in\N$, $[0,1]^n$ is a compact convex set in the Hilbert space $\R^n$. It follows from Theorems \ref{compact metric cp} and \ref{compact convex cp} that $[0,1]^n$ is  $CT_4$. For $n=\infty$,
$ \prod_{k=1}^\infty [0,\tfrac{1}{k}]$
is a compact convex set in the separable Hilbert space $\ell_2$ and is homeomorphic to $[0,1]^\infty$.
Hence. $[0,1]^\infty$ is  $CT_4$.
\end{proof}

\begin{ex} For each $n\in\N$, the n-dimensional unit sphere $\sphere^n$ is  $CT_4$.
\end{ex}
\begin{proof}
We embed $\sphere^n$ into $\R^{n+1}$  endowed with the $(n+1)$-dimensional Euclidean norm. Now, we verify that $\sphere^n$ has the CP. Indeed, for each pair $(A,B)\in\mathcal{T}_4(\sphere^n)$, and for each $x\in \sphere^n$ with $\|x-A\|=\|x-B\|=r>0$. Assume $\|x-B\|=\|x-b\|$ for some $b\in B$. Without loss of generality, we assume that $x=(0,0,\cdots,0,1)$ and $b=(r\sqrt{1-\tfrac{r^2}{4}},0,0,\cdots,0,1-\tfrac{r^2}{2})$.
For each $\varepsilon\in (0,r)$, choose $y=(\varepsilon\sqrt{1-\tfrac{\varepsilon^2}{4}},0,0,\cdots,0,1-\tfrac{\varepsilon^2}{2})\in\sphere$.
Then $\|x-y\|=\varepsilon$. Moreover, we prove that $\|y-a\|>\|y-b\|$ for every $a\in A$, which shows that $x\in\cl \{z\in \sphere\mid \|z-A\|>\|z-B\|\}$. Indeed,
\begin{equation}\label{equ-10}
\begin{array}{ll}
& \|y-b\|^2 \\
=&\left(r\sqrt{1-\tfrac{r^2}{4}}-\varepsilon\sqrt{1-\tfrac{\varepsilon^2}{4}}\right)^2+\left(\tfrac{r^2}{2}-\tfrac{\varepsilon^2}{2}\right)^2\\
=&r^2+\varepsilon^2-2r\varepsilon\sqrt{1-\tfrac{r^2}{4}}\sqrt{1-\tfrac{\varepsilon^2}{4}}-\tfrac{1}{2}r^2\varepsilon^2.\\
\end{array}
\end{equation}
For $a=(a_1,a_2,\cdots,a_{n+1})\in A$, we have that $a_{n+1}\leq 1-\tfrac{r^2}{2}$ and
\[  \begin{array}{ll}
& \|y-a\|^2 \\
=&\left(a_1-\varepsilon\sqrt{1-\tfrac{\varepsilon^2}{4}}\right)^2+a_2^2+\cdots+a_n^2+\left(a_{n+1}-1+\tfrac{\varepsilon^2}{2}\right)^2\\
=&2-2a_1\varepsilon\sqrt{1-\tfrac{\varepsilon^2}{4}}-2a_{n+1}(1-\tfrac{1}{2}\varepsilon^2)\\
=&2-2a_1\varepsilon\sqrt{1-\tfrac{\varepsilon^2}{4}}-2(1-\tfrac{r^2}{2})(1-\tfrac{1}{2}\varepsilon^2)+2(1-\tfrac{r^2}{2}-a_{n+1})(1-\tfrac{1}{2}\varepsilon^2)\\
=&r^2+\varepsilon^2-\tfrac{1}{2}r^2\varepsilon^2 -2a_1\varepsilon\sqrt{1-\tfrac{\varepsilon^2}{4}}+2(1-\tfrac{r^2}{2}-a_{n+1})(1-\tfrac{1}{2}\varepsilon^2).\\
\end{array}
\]
Using Equation (\ref{equ-10}), we have
\begin{equation}\label{equ-20}
\begin{array}{ll}
& \|y-a\|^2 \\
=&\|y-b\|^2 +2r\varepsilon\sqrt{1-\tfrac{r^2}{4}}\sqrt{1-\tfrac{\varepsilon^2}{4}}-2a_1\varepsilon\sqrt{1-\tfrac{\varepsilon^2}{4}}+
2(1-\tfrac{r^2}{2}-a_{n+1})(1-\tfrac{1}{2}\varepsilon^2)\\
=&\|y-b\|^2+2\varepsilon\sqrt{1-\tfrac{\varepsilon^2}{4}}\left(r\sqrt{1-\tfrac{r^2}{4}}-a_1\right)+2(1-\tfrac{r^2}{2}-a_{n+1})(1-\tfrac{1}{2}\varepsilon^2),\\
\end{array}
\end{equation}
where
\[   a_1^2+a_{n+1}^2\leq 1,~ a_{n+1}\leq 1-\tfrac{r^2}{2},~(a_1,a_{n+1})\not =(r\sqrt{1-\tfrac{r^2}{4}},1-\tfrac{r^2}{2}).\]
From Equation~\eqref{equ-20}, a direct calculation shows that $\|y-a\|>\|y-b\|$ for $\varepsilon\in \bigl(0,\min\{r,\tfrac12\}\bigr)$.
\end{proof}

\begin{ex}\label{Cantor-space} The  Cantor space $C$ is  $CT_4$.
\end{ex}

\begin{proof}
Let $C=2^\N$ be the  Cantor space with the metric defined by
\[ d(x,y)=\sum_{n=1}^\infty\frac{|x(n)-y(n)|}{3^n}.\]
 Then for every triple $\{a,x,y\}\subset C $, we have that $d(a,x)\not = d(a,y)$ unless $x=y$. Indeed, let $n(x,y)=\min\{i\mid x(i)\not= y(i)\}$.
 Without loss of generality, suppose that $x(n(x,y))\not=a(n(x,y) )$ and $y(n(x,y))=a(n(x,y) )$. Then
  \begin{gather*}
 d(a,x)-d(a,y)=\frac{1}{3^{n(x,y)}} + \sum_{k=n(x,y)+1}^\infty \frac{|x(k)-a(k)| - |y(k)-a(k)|}{3^k}\\
 \geq \frac{1}{3^{n(x,y)}}-\sum_{k=n(x,y)+1}^\infty\frac{1}{3^k}>0.
 \end{gather*}
  It follows that $(C,d)$ has the CP. Using Theorem \ref{compact metric cp}, we have that $ C$ is $CT_4$.
\end{proof}

\begin{thm}\label{0-dim-CT4} Every compact 0-dimensional metric space is $CT_4$.
\end{thm}

\begin{proof} It is well-known that every compact 0-dimensional metric space $X$ is a subspace of the  Cantor space $C$.
And the  inherited metric from $(C,d)$ also gives (X,d) the CP. Thus, $X$ is $CT_4$.
\end{proof}


\section{Examples of  $CT_3$-spaces which are not  $CT_4$}

Let us recall that a $T_2$-space is called {\bf countably compact} if each countable open cover has a finite subcover. A
$T_2$-space $X$ is countably compact if and only if $X$ does not contain any infinite discrete closed set.  We have the following  theorem:
\begin{thm} \label{non-T4}
Every $CT_4$-space is countably compact.
\end{thm}

\begin{proof} Let $X$ be a $CT_4$-space and $\phi:\mathcal{T}_4(X)\to\Cld(X)^2$ be a  $CT_4$-map for $X$.
If $X$ is not countably compact, then there exists an infinite discrete closed set $D$ in $X$.
Define \[  \mathcal{F}=\{F\subset D\mid F\mbox{ is a non-empty finite set}\}.\]
For each nonempty $S\subset D$, let
\[  \mathcal{F}_S=\{(E,F)\in\mathcal{F}^2\mid E\subseteq S, F\subseteq D\setminus S\}.\]
Then, with the pointwise inclusion relation, $\mathcal{F}_{S}$ is a directed set and hence, $(E,F)\mapsto (E,F)$
is a net in $\mathcal{T}_4(X)$ with the limit $(S,D\sm S)$.
Since $\phi$ is continuous, $\phi(\mathcal{F}_S)\to \phi(S, D\setminus S)$.
Note that $\phi_1(S, D\setminus S)\subseteq X\setminus S  $ and $\phi_2(S, D\setminus S)\subseteq X\setminus (D\setminus S).  $
Since $S$ and $D\setminus S$ are closed sets, $(X\setminus S)^+$ and $(X\setminus (D\setminus S))^+$ are open sets in $\Cld(X)$,
there exists $(E,F)\in \mathcal{F}_S\subseteq \mathcal{F}^2$ such that
$\phi_1(E, F)\subseteq X\setminus S  $ and $\phi_2(E, F)\subseteq X\setminus (D\setminus S).$
Since $\phi_1(E, F)\cup \phi_2(E, F)=X$, we have that $\phi_1(E, F)\cap D=D\setminus S$ and $\phi_2(E, F)\cap D= S$.
Thus, $(E,F)\mapsto \phi_2(E,F)\cap D$ defines a surjective map from $\mathcal{F}^2$ onto the family of all non-empty proper subsets of $D$.
But, it is impossible since $D$ is infinite.
\end{proof}

It is well known that countable compactness coincides with compactness in metric spaces, so we have the following corollary.

\begin{cor}\label{nocomact-no-CT_4} No noncompact metric space is $CT_4$.
\end{cor}

Using Theorems \ref{sum} and  \ref{non-T4}, we immediately obtain the following examples:

\begin{ex} \label{discrete-space}
	All discrete spaces are  $CT_3$. Furthermore, a discrete space is  $CT_4$ if and only if it is finite.
	\end{ex}


We shall give more interesting $CT_3$-spaces which are not $CT_4$.

\begin{lem}\label{continty-map1} For each metric space $(X,d)$,  the map $\varphi: X\times \Cld(X)\to [0,+\infty)$ defined by
	$(x,A)\mapsto d(x,A)$ is continuous.
\end{lem}
\begin{proof}
Recall that  $d(x,A)=\inf_{a\in A}d(x,a).$
For $(x,A)\in X\times \Cld(X)$ and $\varepsilon>0$, choose $a\in A$ such that
\[ d(x,A)\leq d(x,a)<d(x,A)+\tfrac{\varepsilon}{3}.\]
Then
\[ W=B(x,\tfrac{\varepsilon}{3})\times \left((B(A,\tfrac{\varepsilon}{3}))^+\cap (B(a,\tfrac{\varepsilon}{3}))^-\right)\]
is a neighborhood of $(x,A)$, and for each $(y,B)\in W$, choose $b\in B\cap B(a,\tfrac{\varepsilon}{3})$.
Then, on the one hand,
\[ d(y,B)\leq d(y,b)\leq d(y,x)+d(x,a)+d(a,b)< \tfrac{\varepsilon}{3}+d(x,A)+\tfrac{\varepsilon}{3}+\tfrac{\varepsilon}{3}=d(x,A)+\varepsilon.\]
And on the other hand, for each $b'\in B$, choose $a'\in A$ such that $d(a',b')<\tfrac{\varepsilon}{3}$, then
\[ d(y,b')\geq d(x,a')-d(x,y)-d(b',a')>d(x,A)-\tfrac{\varepsilon}{3}-\tfrac{\varepsilon}{3}=d(x,A)-\tfrac{2\varepsilon}{3},\]
and hence,
\[ d(y,B)=\inf_{b'\in B} d(y,b')\geq d(x,A)-\tfrac{2\varepsilon}{3}>d(x,A)-\varepsilon. \]
That is,
\[ |d(y,B)-d(x,A)|<\varepsilon.\]
\end{proof}

\begin{thm}\label{gener-T3} Let $(X,d)$ be a metric space satisfying the following condition: for every $x\in X$ and $r>0$, we have that
\begin{enumerate}
\renewcommand{\labelenumi}{(\arabic{enumi})}
\item[(i)] $\cl \{y\in X\mid d(y,x)>r\}=\{y\in X\mid d(y,x)\ge r\}$;
\item[(ii)]  $\cl \{y\in X\mid d(y,x)<r\}=\{y\in X\mid d(y,x)\le r\}$;
\item[(iii)] $\{y\in X\mid d(y,x)\le r\}$ is compact (hence, $X$ is locally compact);
\item[(iv)] for each $\varepsilon>0$ and $y\in X$ with $r-\varepsilon<d(y,x)<r$, there exists $y_0\in X$ with $d(x,y_0)=r$
and $d(y,y_0)<\varepsilon$.
\end{enumerate}
Then $(X,d)$ is $CT_3$.
\end{thm}
We first verify the following lemma:

\begin{lem}\label{continty-map2} If $(X,d)$ satisfies the requirements in Theorem \ref{gener-T3}, then
$\psi_1,\psi_2:X\times (0, +\infty)\to \Cld(X)$  defined by
		\[ \psi_1(x,r)=\{y\in X\mid d(x,y)\geq r\},~~\psi_2(x,r)=\{y\in X\mid d(x,y)\leq r\}\]
		are continuous.
\end{lem}		
\begin{proof}	We give a complete proof of the continuity of $\psi_1:X\times (0, +\infty)\to \Cld(X)$, and indicate only the modifications needed for $\psi_2$.
	
	Let $(x_0,r_0)\in X\times (0,+\infty)$.

For an open set  $U$  in $X$ with $\psi_1(x_0,r_0)\in U^-,$ by (i),
there exists $y_0\in U$ such that
$d(y_0,x_0)>r_0.$
 Then, there exists $\delta>0$ such that $d(y_0,x)>r$  for $x\in B(x_0,\delta)$ and $r\in (r_0-\delta,r_0+\delta)$. It follows that $U\cap \psi_1(x,r) \ni y_0$. Thus, $\psi_1(x,r)\in U^-$. Hence, $\psi_1^{-1}(U^-)$ is open in $X\times (0,+\infty)$.
		
	For an open set  $U$  in $X$ with $\psi_1(x_0,r_0)\in U^+,$  by (iii), $F(x_0,r_0)=\{y\in X\mid d(y,x_0)=r_0\}$ is compact and $F(x_0,r_0)\subseteq U$.
Thus, there exists $\delta >0$ such that $B(F(x_0,r_0),2\delta )\subseteq U$.
		Now consider neighborhoods $B(x_0,\delta )$ and $ B(r_0,\delta )$ of $x_0$ and $r_0$ in $X$ and $(0,+\infty)$, respectively. For each $x\in B(x_0,\delta )$	and $r\in B(r_0,\delta )$, we check that $\psi_1(x,r)\subseteq U$, which implies that $\psi_1(x,r)\in U^+$. For each $y\in \psi_1(x,r)$, we have that $d(y,x)\geq r$.
		Thus
	\begin{equation}\label{equ-110}
d(y, x_0)> d(y, x)-\delta \geq r-\delta >r_0-2\delta .
\end{equation}
	If $d(y,x_0)\geq r_0$, then $y\in \psi_1(x_0,r_0)\subseteq U$.
If $d(y,x_0)<r_0$, then, by Equation (\ref{equ-110}),
\begin{equation*}
 r_0-2\delta  < d(y, x_0)<r_0.
\end{equation*}
It follows from (iv) that there exists $y_0\in F(x_0,r_0)$ such that $d(y_0,y)<2\delta $.
Hence, $y\in B(F(x_0,r_0),2\delta )\subseteq U$. Therefore, $\psi_1^{-1}(U^+)$ is open.

This shows that $\psi_1:X\times (0,+\infty)\to\Cld(X)$ is continuous.

The continuity of $\psi_2:X\times (0,+\infty)\to\Cld(X)$ has a simpler proof. Indeed, to verify that  $\psi_2^{-1}(U^-)$ is open in $X\times (0,+\infty)$,
we only need to replace ``by (i)'' in the above proof with ``by (ii)'', and to replace ``$d(y_0,x)>r$'' with ``$d(y_0,x)<r$''.

To verify that  $\psi_2^{-1}(U^+)$ is open in $X\times (0,+\infty)$ for each open set $U$ in $X$, we show that every point $(x_0,r_0)$ in $\psi_2^{-1}(U^+)$ is in the interior  of $\psi_2^{-1}(U^+)$. Otherwise, there exist $(x_0,r_0)\in \psi_2^{-1}(U^+) $ and a sequence $(x_n,r_n)$ in $X\times (0,\infty)\setminus \psi_2^{-1}(U^+)$ such that $(x_n,r_n)\to (x_0,r_0)$. Hence, there exists $y_n\in \psi_2(x_n,r_n)\setminus U$ for each $n$. It follows that
\begin{equation}\label{equ-8-30}d(y_n,x_n)\leq r_n.\end{equation}
Since $(x_n,r_n)\to (x_0,r_0)$, we can assume that $d(x_n,x_0)<1$ and $r_n<r_0+1$ for each $n$. Then, $d(y_n,x_0)\leq r_0+2$  for each $ n.$
Using (iii), we can assume that $y_n\to y_0$ for some $y_0\in X$. Then, $y_0\not\in U$. However, it follows from \eqref{equ-8-30} that
$d(y_0,x_0)\leq r_0$. Hence, $y_0\in \psi_2(x_0,r_0)\subset U$. A contradiction occurs.
\end{proof}

\begin{proof}[The proof of Theorem \ref{gener-T3}] For
		$(x,B)\in\mathcal{T}_3(X)$, let
		\begin{gather*}\phi_1(x,B)=\{y\in X\mid d(y,x)\ge \tfrac{1}{2} d(x,B)\}=\cl\{y\in X\mid d(y,x)> \tfrac{1}{2} d(x,B)\},~~\mbox{and}\\
\phi_2(x,B)=\{y\in X\mid d(y,x)\le \tfrac{1}{2} d(x,B)\}=\cl\{y\in X\mid d(y,x)< \tfrac{1}{2} d(x,B)\}.\end{gather*}
Then $x\not\in \phi_1(x,B)$, $B\cap \phi_2(x,B)=\emptyset$ and $\phi_1(x,B)\cup \phi_2(x,B)= X.$ Moreover,
\[ \phi_1(x,B)=\psi_1(x,\tfrac{1}{2}d(x,B)),~~\phi_2(x,B)=\psi_2(x,\tfrac{1}{2}d(x,B)).\]
It follows from Lemmas \ref{continty-map1} and \ref{continty-map2} that $\phi_1,\phi_2:\mathcal{T}_3(X)\to \Cld(X)$
		are continuous.
Hence $X$ is $CT_3$.
	\end{proof}

\begin{ex}\label{Euclidean-space} Every Euclidean space $\R^n$ is $CT_3$ but not $CT_4$.
\end{ex}

\begin{proof}It is easy to verify the Euclidean metric on $\R^n$ satisfies the requirements in Theorem \ref{gener-T3}. Hence, $\R^n$ is $CT_3$.
By Corollary \ref{nocomact-no-CT_4},  $\R^n$ is not $CT_4$.
\end{proof}

\begin{rem} We do not know whether or not there exist other metric spaces satisfying the requirements in Theorem \ref{gener-T3}.\end{rem}

\begin{ex}\label{real-numbers} Every interval in $\R$ is  $CT_3$. And an interval in $\R$ is  $CT_4$ if and only if it is compact.
\end{ex}

\begin{proof}Obviously, every interval in $\R$ is homeomorphic to one of $\R$, $[0,1]$, $\R^+=[0, +\infty)$ or $\{0\}$.
Hence, by Corollary \ref{nocomact-no-CT_4} and Examples \ref{discrete-space}, \ref{cube} and \ref{Euclidean-space}, we only need to verify that $\R^+$ is $CT_3$.

Let \begin{gather*}
\mathcal{L}=\{(x,B)\in\mathcal{T}_3(\R^+)\mid x<b\mbox{ for each }b\in B\};\\
\mathcal{R}=\{(x,B)\in\mathcal{T}_3(\R^+)\mid x>b\mbox{ for each }b\in B\};\\
\mathcal{M}=\mathcal{T}_3(\R^+)\setminus (\mathcal{R}\cup\mathcal{L}).
\end{gather*}
Then, for each $(x,B)\in \mathcal{L}$, choose $a\in\R^+$ such that $a>x$ and $(a,+\infty)\supseteq B$. Then $[0,a)\times (a,+\infty)^+$
is a neighborhood of $(x,B)$ in $\mathcal{T}_3(\R^+)$ and  misses $\mathcal{R}\cup \mathcal{M}$. Hence, $\mathcal{L}$ is open in $\mathcal{T}_3(\R^+)$.
Similarly, $\mathcal{R}$ is open in $\mathcal{T}_3(\R^+)$. Moreover, for each $(x,B)\in \mathcal{M}$, choose $a,c\in\R^+$,  such that $a<x<c$ and $B\subseteq [0,a)\cup (c,+\infty)$.  Then $(a,c)\times \left([0,a)^-\cap (c,+\infty)^-\right)$ is a neighborhood of $(x,B)$ in $\mathcal{T}_3(\R^+)$ and misses $\mathcal{L}\cup \mathcal{R}$.
This shows that $\mathcal{M}$ is also open in $\mathcal{T}_3(\R^+)$. Thus, $\mathcal{L},\mathcal{R},\mathcal{M}
$ are clopen sets in $\mathcal{T}_3(\R^+)$.

It is easy to see that $(x,B)\mapsto \min B$ is continuous from $\mathcal{L}$ to $\R^+$. Thus, we can define a continuous map
$\phi:\mathcal{L}\to \Cld(\R^+)^2$ as follows
\[ \phi(x,B)=\left([\tfrac{x+\min B}{2},+\infty),[0,\tfrac{x+\min B}{2}]\right).\]
Then $x\not\in \phi_1(x,B)$, $B\cap \phi_2(x,B)=\emptyset$ and $ \phi_1(x,B)\cup\phi_2(x,B)=\R^+.$

Similarly, we can define a continuous map $\phi:\mathcal{R}\cup \mathcal{M}\to \Cld(\R^+)^2$ satisfying the above conditions.
Thus, $\R^+$ is $CT_3$.
\end{proof}


The following theorem gives a class of $CT_3$-spaces:

\begin{thm}\label{metric+one-noniso-be-CT3}
Let $(X,d)$ be a metric space with a unique non-isolated point $x_0$. Then $X$ is $CT_3$, and it is $CT_4$ if and only if it is compact.
\end{thm}

\begin{proof}
We may replace the metric $d$ by a topologically equivalent metric $d_1$ on $X$ satisfying $d_1(x,x_0)<1$ for all $x\in X$ and
\[
B_{d_1}\!\left(x_0,\tfrac1n\right)\setminus B_{d_1}\!\left(x_0,\tfrac1{n+1}\right)\neq\emptyset
\]
for every $n\in\mathbb{N}$. We henceforth assume $d=d_1$ without loss of generality.
	
For every $n\in\N$, let
\[
\begin{array}{lll}
\mathcal{T}_3^{d}(X)&=& \{(x,B)\in \mathcal{T}_3(X)\mid d(b_0,x_0)<\frac{1}{n}\leq d(x,x_0) \mbox{ for some } b_0\in B, n\in\N\}, \\
 \mathcal{T}_3^{u,n}(X)&=&\{(x,B)\in \mathcal{T}_3(X)\mid \frac{1}{n+1}\leq d(x,x_0)< \frac{1}{n} \mbox{ and } \\
&&~~d(b,x_0)\geq \frac{1}{n+1}  \mbox{ for each $b\in B$ and  }d(b_0,x_0)< \frac{1}{n} \mbox{ for some }b_0\in B \},\\
\mathcal{T}_3^{u,n,\infty}(X)&=& \{(x,B)\in \mathcal{T}_3(X)\mid d(x,x_0)<\frac{1}{n+1}, d(b,x_0)\geq \frac{1}{n+1}\\
 && ~~\mbox{ for each $b\in B$ and  }d(b_0,x_0)< \frac{1}{n} \mbox{ for some }b_0\in B \}.
\end{array}
\]
Then $\{\mathcal{T}_3^{d}(X),\mathcal{T}_3^{u,n}(X),\mathcal{T}_3^{u,n,\infty}(X)\mid n\in\N\}$ is a partition of $\mathcal{T}_3(X)$.
Moreover, we verify that each element in the family is open in $\mathcal{T}_3(X)$, which shows that the partition consists of clopen sets. For each $(x,B)\in \mathcal{T}_3^{d}(X)$, there exists
$n\in\N$ such that $d(b_0,x_0)<\frac{1}{n}\leq d(x,x_0)$ for some $ b_0\in B$.
Then
\[
 (x,B)\in\{x\}\times B(x_0,\tfrac{1}{n})^-\subseteq \mathcal{T}_3^{d}(X).
\]
Hence, $\mathcal{T}_3^{d}(X)$ is open. For each $(x,B)\in \mathcal{T}_3^{u,n}(X)$,
\[
 (x,B)\in\{x\}\times \left( B(x_0,\tfrac{1}{n})^-\cap (X\sm B(x_0,\tfrac{1}{n+1}))^+\right)\subseteq \mathcal{T}_3^{u,n}(X).
\]
Thus, $\mathcal{T}_3^{u,n}(X)$ is open. And for each $(x,B)\in \mathcal{T}_3^{u,n,\infty}(X)$,
\[
 (x,B)\in B(x_0,\tfrac{1}{n+1})\times \left( B(x_0,\tfrac{1}{n})^-\cap (X\sm B(x_0,\tfrac{1}{n+1}))^+\right)\subseteq \mathcal{T}_3^{u,n,\infty}(X).
\]
Therefore, $\mathcal{T}_3^{u,n,\infty}(X)$ is also open.

 Define $\phi:\mathcal{T}_3(X)\to \Cld(X)^2$ as follows, for $(x,B)\in \mathcal{T}_3(X)$,
\[
\phi(x,B)=\left\{
\begin{array}{ll}
(X\sm\{x\},\{x\}) & (x,B)\in \mathcal{T}_3^d(X)\cup \mathcal{T}_3^{u,n}(X),\\
(X\sm B(x_0,\tfrac{1}{n+1}), B(x_0,\frac{1}{n+1})) & (x,B)\in \mathcal{T}_3^{u,n,\infty}(X).
\end{array}
\right.
\]
For each $(x,B)\in \mathcal{T}_3^d(X)\cup \mathcal{T}_3^{u,n}(X)$, $x$ is an isolated point and hence, $\phi(x,B)\in\Cld(X)^2$. Moreover,
$\phi$ is constant  on the open set $(\{x\}\times \Cld(X))\cap (\mathcal{T}_3^d(X)\cup \mathcal{T}_3^{u,n}(X))$ in $\mathcal{T}_3(X)$.
It follows that $\phi$ is continuous on $\mathcal{T}_3^d(X)\cup \mathcal{T}_3^{u,n}(X)$. Obviously, it also satisfies the other conditions. Therefore,
$\phi$ is a $CT_3$-map on $\mathcal{T}_3^d(X)\cup \mathcal{T}_3^{u,n}(X)$.
Further, for each $n\in\N$ and $(x,B)\in  \mathcal{T}_3^{u,n,\infty}(X)$, $\phi(x,B)$ is in $\Cld(X)^2$, and $\phi$ is constant  on the open set
$\mathcal{T}_3^{u,n,\infty}(X)$. Hence, it is continuous on $\mathcal{T}_3^{u,n,\infty}(X)$ and, moreover, is a $CT_3$-map.

Therefore, $X$ is $CT_3$.

The second statement follows from Theorems \ref{non-T4} and  \ref{0-dim-CT4}.
\end{proof}
\begin{rem} Example \ref{one-point} shows that the metrizability assumption in the above theorem cannot be replaced by compactness.
Theorem \ref{CT3-imp-second} shows that, for a countable space $X$ with a unique non-isolated point, $X$ is $CT_3$ if and
only if $X$ is metrizable.
\end{rem}
For an infinite cardinal $m$, we identify $m$ with the least ordinal of cardinality $m$; thus, $m$ is the set of all ordinals smaller than $m$.
Let
 \[ H(m)=\{\infty\}\cup\{(n,\alpha )\mid  n\in\mathbb{N},\alpha <m\}\]
 with the metric:
\[ d(x,y)=\left\{
\begin{array}{ll}
|\tfrac{1}{n_1}-\tfrac{1}{n_2}| & \mbox{if } x=(n_1,\alpha ), y=(n,\alpha ),\\
\tfrac{1}{n} & \mbox{if } x=(n,\alpha ), y=\infty \mbox{ or }x=\infty, y=(n,\alpha ),\\
\tfrac{1}{n_1}+\tfrac{1}{n_2} & \mbox{if } x=(n_1,\alpha _1), y=(n_2,\alpha _2)\mbox{ with }\alpha _1\not=\alpha _2.
\end{array}\right.
\]
Then $(H(m),d)$ is a metric space with the unique non-isolated point $\infty$.
That is, $(H(m),d)$ is a subspace of the hedgehog $J(m)$ of spininess $m$, see \cite[Example 4.1.5]{Engelking-1989}.
Using Theorem \ref{metric+one-noniso-be-CT3}, we have the following example:

\begin{ex}\label{hedgehog} The space $(H(m),d)$ is $CT_3$ but not $CT_4$.
\end{ex}



\section{Examples of  $CT_2$-spaces which are not  $CT_3$}

We have the following interesting result:

\begin{thm}\label{CT3-imp-second} If $X$ is a $CT_3$-space, then $w(Y)= d(Y)$ for each  subspace $Y$  of $X$. In particular, every separable subspace of a $CT_3$-space is metrizable.
\end{thm}

\begin{proof} Without loss of generality, we can assume that $Y$ is a closed subspace of $X$. Then, $\mathcal{T}_3(Y)$ is a subspace of $\mathcal{T}_3(X)$.

  Let $\phi$ be a $CT_3$-map for $X$ and let $D$ be a dense set in $Y$ with $|D|=d(Y)$. We only consider the case that $d(Y)$ is infinite.
  Define
  \[ \mathcal{F}=\{(d,F)\in \mathcal{T}_3(Y)\mid \{d\}\cup F \mbox{ is a finite set in } D\}. \]
  Then, $|\mathcal{F}|=d(Y)$. For each $(y,B)\in \mathcal{T}_3(Y)$ and each neighborhood $U\times \langle V_1,V_2,\cdots,V_n\rangle$ of $(y,B)$ in $\mathcal{T}_3(Y)$, where $U$ and $V_i$ are open in $Y$
  for $i\leq n$, we can assume that $U\cap V_i=\emptyset$ for $i\leq n$. Choose $d\in D\cap U$ and $d_i\in D\cap V_i$ for $i\leq n$.
  Then $(d,\{d_1,d_2,\cdots, d_n\})\in \mathcal{F}\cap \left(U\times \langle V_1,V_2,\cdots,V_n\rangle\right)$. It follows that $\mathcal{F}$ is dense
  in $\mathcal{T}_3(Y)$.

  Let
  \[ \mathcal{B}=\{Y\sm \phi_1(d,F)\mid (d,F)\in\mathcal{F}\}.\]
 Then $\mathcal{B}$ is a family of open sets in $Y$ and $| \mathcal{B}|\leq d(Y)$. To show $w(Y)\leq d(Y)$, it suffices to verify that $\mathcal{B}$ is a base for $Y$.

  For each $y\in Y$ and for  each open neighborhood $U$ of $y$ in $Y$, let $B=Y\setminus U$.
  Then $(y,B)\in \mathcal{T}_3(Y)$. Then, we have that $\phi_1(y,B)\not\ni y$ and $\phi_2(y,B)\cap B=\emptyset$.
That is,
\[ \phi(y,B)\in (X\sm\{y\})^+\times (X\sm B)^+.\]
 By the continuity of $\phi$, $\phi^{-1}\left((X\sm\{y\})^+\times (X\sm B)^+\right)$ is an  open set in $\mathcal{T}_3(X)$.
 Notice that
 \[ (y,B)\in \mathcal{T}_3(Y)\cap \phi^{-1}\left((X\sm\{y\})^+\times (X\sm B)^+\right).\]
 It follows from the density of $\mathcal{F}$ in $\mathcal{T}_3(Y)$ that we can choose
 \[ (d,F)\in\mathcal{F}\cap \phi^{-1}\left((X\sm\{y\})^+\times (X\sm B)^+\right).\]
 Thus, $(d,F)\in\mathcal{F}$ and
 $\phi_1(d,F)\in (X\setminus\{ y\})^+$ and $\phi_2(d,F)\in (X\setminus B)^+.$
That is,
\[ \phi_1(d,F)\not\ni y,~~\phi_2(d,F)\cap B=\emptyset.\]
It follows from $\phi_1(d,F)\cup \phi_2(d,F)=X$ that
\[ y\in Y\setminus \phi_1(d,F)=Y\cap (X\setminus \phi_1(d,F))\subseteq Y\cap \phi_2(d,F)\subseteq Y\setminus B=U.\]
Note that $Y\setminus \phi_1(d,F)\in\mathcal{B}$. We have that $\mathcal{B}$
is a base for $Y$.
\end{proof}

\begin{rem} Example \ref{omega1} shows that the converse of the first statement in Theorem \ref{CT3-imp-second}
does not hold; however, we do not know whether or not the converse of the second statement holds, that is, whether every separable metrizable space is $CT_3$.\end{rem}

\begin{cor} Any space containing $2^{\omega_1}$ as a subspace is not $CT_3$, where $\omega_1$ is the first uncountable cardinal and $2^{\omega_1}$
is the product space of the two-point spaces.
\end{cor}

\begin{ex} For every uncountable cardinal $m$, neither $\II^m$ nor $2^m$ is $CT_3$.
\end{ex}

 The following theorem gives a class of $CT_2$-spaces:

\begin{thm}\label{CT_2notCT3}
 Let $X=D\cup\{x\}$ be  a  space, where $D$ is a discrete subspace. If $X$ satisfies the condition that
 \begin{equation}\label{equ-711}
 \begin{split}\mbox{there exists a linear order $\leq$ on $D$ such that} \\
 \mbox{ every upper-bounded subset in $(D,\leq )$ is closed in $X$,}
 \end{split}
 \end{equation}
 then $X$ is $CT_2$.
\end{thm}

\begin{proof}  We can extend $\leq $ to $X$ such that $x$ is the largest element in $(X,\leq)$. For every $d\in D$,
by our assumption, the set
\[ \uparrow d=\{y\in X\mid d< y\}\]
 is a clopen set in $X$. Thus,
\[ \mathcal{T}^{b,d}_2(X)=\{d\}\times \uparrow d,~\mathcal{T}^{u,d}_2(X)=\uparrow d\times\{d\}. \]
are clopen sets in $\mathcal{T}_2(X)$. It follows that $\{\mathcal{T}^{b,d}_2(X),\mathcal{T}^{u,d}_2(X)\mid d\in D\}$ is a partition of $\mathcal{T}_2(X)$ consisting of clopen sets.

For $(d,y)\in\mathcal{T}_2^{b,d}(X)$, let
\[ \phi(d,y)=\left( X\setminus \{d\},\{d\}\right).\]
And for $(y,d)\in\mathcal{T}_2^{u,d}(X)$, let
\[ \phi(y,d)=\left( \{d\},X\setminus \{d\}\right).\]
It is not hard to verify that, for each $(d,y)\in\mathcal{T}_2^{b,d}(X)$, $\phi(d,y)\in\Cld(X)^2$, and $d\not\in \phi_1(d,y)$, $y\not\in \phi_2(d,y)$,
 $\phi_1(d,y)\cup \phi_2(d,y)=X$.
Moreover,
$\phi$ is constant on  $\mathcal{T}_2^{b,d}(X)$, and hence $\phi$ is continuous on $\mathcal{T}_2^{b,d}(X)$.
Therefore, $\phi$ is a $CT_2$-map on $\mathcal{T}_2^{b,d}(X)$. Similarly, $\phi$ is a $CT_2$-map on $\mathcal{T}_2^{u,d}(X)$.
Therefore, $X$ is $CT_2$.
\end{proof}

\begin{rem} It is trivial to observe that Condition (\ref{equ-711}) holds if $D$ is countable. However, Example \ref{one-point} shows that Condition (\ref{equ-711}) in  Theorem \ref{CT_2notCT3} is essential.
\end{rem}

Let $\beta D$ be the C\'{e}ch-Stone compactification of an infinite discrete space $D$ and $p\in \beta D\setminus D$. As a subspace of $\beta D$,
 $D\cup\{p\}$ is a  space with the unique non-isolated point $p$. We have the following result:

 \begin{lem}[\cite{Comfort},
\cite{Comfort-Negrepontis}]\label{Stone-compactification} For each $p\in \beta \N\setminus \N$, $d(\N\cup\{p\})<w(\N\cup\{p\})$. For each infinite discrete space $D$, there exists $p\in \beta D\setminus D$
 such that $d(D\cup\{p\})<w(D\cup\{p\})$.
 \end{lem}

 Let $I$ be an infinite set and for each $i\in I$, let $S_i=\{x_{i,j}\mid j=1,2,\cdots,\infty\}$ be a convergent sequence with the limit $x_{i,\infty}$. And let
	$S_I=\bigoplus_{i\in I}S_i/\{x_{i,\infty}\mid i\in I\}$ be the quotient space of the topological
	sum $\bigoplus_{i\in I}S_i$ by collapsing  all the limit points to
	 a single point.
	We call this space a {\bf sequence hyperfan}\footnote{As is well known, when $I$ is countable, $S_I$ is called the {\bf sequence fan}. We use the term {\bf sequence hyperfan} for arbitrary $I$.}. Then $S_I$ is a space with the unique non-isolated point.
Using Theorems \ref{CT3-imp-second} and \ref{CT_2notCT3}, we have the following examples:

\begin{ex}\label{ex-one-pt-com-uncountable-discrete} For every infinite discrete space $D$ and for every $p\in \beta D\setminus D$,  the space $D\cup\{p\}$ is $CT_2$. However,  there exists
$p\in \beta D\setminus D$ such that the space $D\cup\{p\}$ is  not $CT_3$. And every sequence hyperfan is $CT_2$ but not $CT_3$.
\end{ex}

\begin{proof}  For every $p\in \beta D\setminus D$, there exists a subset $D'$ of $D$ such that $p\in \cl D'$ and $p\not\in \cl E$ for each $E\subset D'$ with $|E|<|D'|$. Define a well-order $\leq $ on $D$ such that $d<d'$ if $d\in D\sm D'$ and $d'\in D'$, and the order-type of $(D',\leq)$ is $|D'|$.
Then, for each bounded set $A$ in $D$, there exists $d'_0\in D'$ such that $A\subseteq (D\sm D')\cup \{d'\in D'\mid d��<d'_0\}$. And
$p\not \in \cl \left((D\sm D')\cup \{d'\in D'\mid d��<d'_0\}\right)$, we have that $p\not\in A$. That is, $D\cup\{p\}$
satisfies Condition (\ref{equ-711}) in Theorem \ref{CT_2notCT3}. Hence, $D\cup\{p\}$ is $CT_2$.

Using Theorem \ref{CT3-imp-second} and Lemma \ref{Stone-compactification}, there exists
$p\in \beta D\setminus D$ such that the space $D\cup\{p\}$ is  not $CT_3$.

To show that $S_I$ is $CT_2$, let $D=\{x_{i,j}\mid i\in I,j=1,2,\cdots\}.$ Then $D$ is discrete.  For every well-order $\leq $ on $I$, we can define a
well-order $\leq $ on $D$ as
\[ x_{i,j}\leq x_{i',j'}\mbox{ if and only if either $j<j'$, or $j=j'$ and $i\leq i'$. } \]
Hence, $S_I$ satisfies Condition (\ref{equ-711}) in Theorem \ref{CT_2notCT3}. It follows from Theorem \ref{CT_2notCT3} that $S_I$ is $CT_2$.

Now we show that $w(S_I)>|I|$.

Let $p$ be the unique non‑isolated point of $S_I$. Suppose for contradiction that $w(S_I)\le |I|$. Since the weight bounds the minimal cardinality of a local base at any point, there exists a local base $\mathcal{B}=\{B_\alpha\mid \alpha<\kappa\}$ at $p$ with $\kappa\le |I|$.

For each $\alpha<\kappa$ and each $i\in I$, since $B_\alpha$ is a neighbourhood of $p$, by the quotient topology of $S_I$, $B_\alpha$ contains all but finitely many points of the $i$-th sequence. Thus we may choose $n(i,\alpha)\in\mathbb{N}$ such that $x_{i,j}\in B_\alpha$ whenever $j\ge n(i,\alpha)$.

Fix an injection $\varphi\colon \kappa\to I$, and write $i_\alpha=\varphi(\alpha)$ for each $\alpha<\kappa$. Set $k_\alpha = n(i_\alpha,\alpha)$, and let $A=\{x_{i_\alpha,k_\alpha}\mid \alpha<\kappa\}$. Define $U=S_I\setminus A$.

For every $i\in I$, $U$ removes at most one point from the $i$-th sequence, so $U$ intersects each sequence in a co‑finite subset. By the quotient topology, $U$ is an open neighbourhood of $p$.

Now take any $B_\alpha\in\mathcal{B}$. We have $x_{i_\alpha,k_\alpha}\in B_\alpha$ by the choice of $k_\alpha$, while $x_{i_\alpha,k_\alpha}\notin U$. Hence $B_\alpha\not\subseteq U$, which contradicts that $\mathcal{B}$ is a local base at $p$.

Therefore $w(S_I)>|I|$.

It is straightforward to verify $d(S_I)=|I|$, so $w(S_I)>d(S_I)$.
	
By Theorem \ref{CT3-imp-second}, $S_I$ is not  $CT_3$.
\end{proof}

We give a property of $CT_3$-spaces which allows us to obtain a class of $CT_2$-spaces which are not $CT_3$.

Let us recall that a Hausdorff space $X$ is called a  {\bf Fr\'{e}chet-Urysohn\ space} (which is called a {\bf Fr\'{e}chet\ space} in \cite{Engelking-1989}) if for each set $A$ in $X$,  every accumulation point $x$ of $A$ is the limit  of  a  sequence in $A$. A Hausdorff topological space is called a {\bf  strict Fr\'{e}chet-Urysohn space}  if
for each sequence $(A_n)$ of sets in $X$ such that $x$ is an accumulation point of each $A_n$, there exists a sequence $(x_n)$ in $X$
such that $x_n\in A_n$ and $x_n\to x$.

\begin{thm}\label{thm-unique-non-iso-2}
Every $CT_3$-space is  a strict Fr\'{e}chet-Urysohn\ space.
\end{thm}

\begin{proof}
Let $\phi:\mathcal{T}_3(X)\to \Cld(X)^2$ be a $CT_3$-map for $X$.

We first show that $X$ is a Fr\'{e}chet-Urysohn space.

Assume $Y\subseteq X\sm\{x_0\}$ and $x_0\in\cl Y$. Let
\[
\begin{split}
\mathcal{C}=\{\0\not=C\subseteq Y\mid x_0\not\in \cl C\},\\
\mathcal{F}=\{F\in\mathcal{C}\mid F\mbox{ is finite}\}.
\end{split}
\]
Then for each $C\in\mathcal{C}$, $\phi_2(x_0,\cl C)\cap \cl C=\emptyset$. Hence, there exists $F_C\in\mathcal{F}$ such that $F_C\subseteq C$ and
\begin{equation}\label{equ-313.5}
\phi_2(x_0,A)\cap \cl C=\emptyset,~\mbox{ for }~A\in\Cld(X) \mbox{ with }~F_C\subseteq A\subseteq \cl C.
\end{equation}

Now, for each $F\in\mathcal{F}$, let
\[
\mathcal{C}_F=\{C\in \mathcal{C}\mid F_C\subseteq F\}\ni F.
\]
We show the following \vspace*{3mm} claim:\\
{\bf Claim 1.} There exists no $F_0\in \mathcal{F}$ such that, for each $A\in \mathcal{F}$, there exists $C\in \mathcal{C}_{F_0}$ such that $A\subseteq C$.

 Otherwise, assume that $F_0\in \mathcal{F}$ satisfies the requirement and let
\[
\mathcal{C}_{F_0}^0=\{C\in \mathcal{C}_{F_0}\mid C\supseteq F_0\}.
\]
Then,
\begin{equation}\label{equ-203.5}
\bigcup\mathcal{C}_{F_0}^0=Y.
\end{equation}
Indeed, for each $y\in Y$, by definition of $\mathcal{F}$,  $\{y\}\cup F_0\in\mathcal{F}$. Thus, there exists $C\in\mathcal{C}_{F_0}$ such that $\{y\}\cup F_0\subseteq C$ and hence $C\in\mathcal{C}_{F_0}^0$. Hence, $y\in \bigcup\mathcal{C}_{F_0}^0$.

  Moreover, for each $C\in \mathcal{C}_{F_0}^0$, by Equation (\ref{equ-313.5}),
\[
\phi_2(x_0,F_0)\cap\cl C=\emptyset.
\]
It follows from Equation (\ref{equ-203.5}) that $\phi_1(x_0,F_0)\supseteq Y$. By the definition of  $CT_3$-map and $x_0\in\cl Y$,
	we obtain $x_0\in \phi_1(x_0,F_0)$, a contradiction. This completes the proof of Claim 1.

For each $F\in\mathcal{F}$, let
\[
\mathcal{D}_{F}=\{D\in\mathcal{C}\mid D\not\subseteq C\mbox{ for each }C\in \mathcal{C}_F\}.
\]
It follows from Claim 1 that $\mathcal{D}_{F}$ has the \vspace*{3mm} following properties:\\
{\bf Claim 2.}
\begin{enumerate}
\item [(1)]\label{convex1.5}$\mathcal{C}_F\cap \mathcal{D}_{F}=\emptyset.$
\item [(2)]\label{convex2} $\mathcal{D}_{F}\cap \mathcal{F}\not=\emptyset$.
\item [(3)]\label{convex2.5}  $\mathcal{D}_{F}\supseteq\mathcal{D}_{F'}$ if $\0\not=F\subseteq F'\in\mathcal{F}$.
\item [(4)]\label{convex2.6} $\mathcal{C}\ni D'\supseteq D\in \mathcal{D}_{F}$ implies that $D'\in \mathcal{D}_{F}$.
\end{enumerate}
{\bf Claim 3.} Let
\[
\mathcal{D}=\{D\in\mathcal{C}\mid D\in\mathcal{D}_F \mbox{ for each } F\in \mathcal{F} \mbox{ with } F\subseteq D\}.
\]
Then
\begin{equation}\label{equ-210.5}
\mathcal{D}=\emptyset.
\end{equation}

Otherwise, assume that $D\in \mathcal{D}$. Then, $F_D\in\mathcal{F}$ and $F_D\subseteq D$. It follows that $D\in\mathcal{D}_{F_D}$.
However, by the definition of $\mathcal{C}_{F_D}$, we have that $D\in\mathcal{C}_{F_D}$. Hence, $D\in \mathcal{C}_{F_D}\cap \mathcal{D}_{F_D}$, which
contradicts (1) in Claim 2.

Choose $F_1\in \mathcal{F}$. Then, by (2) in Claim 2, choose $E_1\in \mathcal{D}_{F_1}\cap \mathcal{F}$ and let $F_2=E_1\cup F_1$.
It follows from (4) in Claim 2 that $F_2\in \mathcal{D}_{F_1}\cap \mathcal{F}$.
Moreover, choose  $E_2\in \mathcal{D}_{F_2}\cap \mathcal{F}$ and let $F_3=E_2\cup F_2$. Proceeding in this way, we can define a sequence $(F_1\subseteq F_2\subseteq\cdots)$ in $\mathcal{F}$ satisfying the condition  that
\begin{equation}\label{equ-215.5}
F_{n+1}\in \mathcal{D}_{F_{n}}\sm \mathcal{D}_{F_{n+1}},~~n=1,2,\cdots.
\end{equation}
Let $A=\bigcup_{n=1}^\infty F_n$. Then, $A$ is a countable set in $Y$. Moreover, for each $F\in \mathcal{F}$ with $ F\subseteq A$, there exists $n$ such that $F\subseteq F_n\subseteq F_{n+1} \subseteq A$. It follows from Equation (\ref{equ-215.5}) and (3) in Claim 2 that
\[
A\supseteq F_{n+1}\in \mathcal{D}_{F_n} \subseteq \mathcal{D}_{F}.
\]
Using the definition of $\mathcal{D}$ and (4) in Claim 2, we have $A\in\mathcal{D}$ unless $A\not\in  \mathcal{C}$. By Equation (\ref{equ-210.5}), we must  have that $A\not\in  \mathcal{C}$,
 that is, $x_0\in\cl A$.
Thus, by Theorem \ref{CT3-imp-second}, $\cl A$ is a metrizable subspace of $X$ and $A\subseteq Y$. Thus, there exists a sequence $(y_n)$ in $A\subseteq Y$ such that $y_n\to x_0$.

Second, using the above statement, we show that $X$ is a strict Fr\'{e}chet-Urysohn space. Assume that $(A_n)$ is a  sequence of sets in $X\sm\{x_0\}$ with $x_0$ being an accumulating point of each $A_n$. Then, by the first statement, for each $n$, there exists a countable subset $C_n$ of $A_n$ such that $x_0\in\cl C_n$. Let
\[
C=\bigcup_{n=1}^\infty C_n.
\]
Then $C\cup \{x_0\}$ is countable. It follows from Theorem \ref{CT3-imp-second} that $C\cup \{x_0\}$ is metrizable. Thus, there exists a sequence  $(x_n)$ in $C$
such that $x_n\in C_n\subseteq A_n$ and $x_n\to x_0$. We are done.
\end{proof}

Recall that {\bf cofinality} $\cf (\alpha)$ of an ordinal $\alpha$ is defined by
\[\cf (\alpha)=\min\{\gamma\mid \mbox{ there exists } \{\xi_\eta\mid \eta<\gamma\}\subseteq\alpha
\mbox{ such that }\sup\xi_\eta=\alpha\}.\]

For each infinite ordinal $\alpha$, let $\alpha^+=\alpha\cup\{\alpha\}$ \footnote{To distinguish this space from the standard successor ordinal space, we denote this space by $\alpha^+$ and reserve $\alpha+1$ for the standard ordinal space.} with the topology defined by
\[
\mbox{a set $A$ in $\alpha^+$ is closed if either $A\ni \alpha$ or $A$ is bounded in $\alpha$.}
\]
Then $\alpha^+$ is $CT_2$, by Theorem \ref{CT_2notCT3}. Moreover, we have the following theorem:

\begin{thm} For an infinite ordinal $\alpha$,   $\alpha^+$ is $CT_3$ if and only if it is metrizable if and only if $\cf (\alpha)$ is countable.
\end{thm}

\begin{proof} Trivially, $\alpha^+$ is first countable if and only if $\alpha^+$ is metrizable if and only if  $\cf (\alpha)$ is countable.
 Observe that $\alpha^+$ has exactly one non-isolated point.  Hence, if $\cf(\alpha)$ is countable, that is, if $\alpha^+$ is metrizable, then by Theorem \ref{metric+one-noniso-be-CT3}, $\alpha^+$ is $CT_3$. On the other hand, if $\alpha^+$ is $CT_3$, then by Theorem \ref{thm-unique-non-iso-2},  $\alpha^+$ is a Fr\'{e}chet-Urysohn\ space, which implies that $\cf(\alpha)$ is countable.
 \end{proof}

\begin{cor} For an infinite ordinal $\alpha$, $\alpha^+$ is  $CT_2$  but not $CT_3$  if $\cf (\alpha)$ is not countable.
\end{cor}

Moreover, we have the following example:

\begin{ex} The Sorgenfrey line  $\R_l$ is $CT_2$  but not $CT_3$.
\end{ex}

\begin{proof} \leavevmode We use the upper-limit Sorgenfrey topology on $\mathbb R$, generated by half-open intervals $(a,b]$.
 It is well known that $\R_l$ is separable but not metrizable. Hence, by Theorem \ref{CT3-imp-second}, $\R_l$ is not $CT_3$.

 Define
$l,r:\R_l\to \Cld(\R_l)$ as follows
\[ l(x)=(-\infty,x],~~r(x)=[x,+\infty).\]
We show that they are continuous. Since $\R_l$ is first-countable, to this end, assume that $x_n\to x$ in $\R_l$ with $x_n\leq x_{n+1}\leq x$ for each $n$; we verify that
$l(x_n)\to l(x)$ and $r(x_n)\to r(x)$.
Indeed, $\{l(x_n)\}$ is increasing and hence
\[ \lim_{n\to\infty}l(x_n)=\cl \left(\bigcup_{n\in\N}l(x_n)\right)=l(x).\]
To show $r(x_n)\to r(x)$, suppose that $r(x)\in U^+$ for some open set $U$ in $\R_l$, there exists $(a,b]$ such that $x\in (a,b]\subseteq U$.
Hence, there exists $N\in\N$ such that $x_n\in (a,b]$ for $n>N$. It follows that
\[ r(x_n)=[x_n,x]\cup r(x)\subseteq U.\]
That is, $r(x_n)\in U^+$ for $n>N$. Since $\{r(x_n)\}$ is decreasing, we have that $r(x_n)\to r(x)$.

Note that $\mathcal T_2(\mathbb R_l)=\{(x,y)\in\mathbb R_l^2:x\neq y\}$, and the subsets $\{(x,y):x<y\}$ and $\{(x,y):x>y\}$ are clopen in $\mathbb R_l^2\setminus\Delta(\mathbb R_l)$.
Now, for each $(x,y)$ in $\R_l^2$ with $x<y$, define
\[ \phi(x,y)=\left([\tfrac{x+y}{2},+\infty), (-\infty,\tfrac{x+y}{2}]\right).\]
Similarly, for $(x,y)$ in $\R_l^2$ with $x>y$, define
	\[ \phi(x,y)=\left((-\infty,\tfrac{x+y}{2}],[\tfrac{x+y}{2},+\infty)\right).\]
It is well-known that $(x,y)\mapsto \tfrac{x+y}{2}$ is continuous from $\R_l^2$ to $\R_l$. It follows that $\phi$ is continuous.
Trivially, it also satisfies the other requirement. Hence, $\R_l$ is $CT_2$.
\end{proof}

\section{Examples of  $CT_1$-spaces which are not  $CT_2$}

\begin{ex} \label{T_1-not-T_2} Each infinite set $X$ endowed with the co-finite topology $\mathcal{T}_f$ is  $CT_1$ but not  $CT_2$.
\end{ex}
\begin{proof}
	It is well known that $(X, \mathcal{T}_f)$ is $T_1$ but not $T_2$; therefore, it is not  $CT_2$ either and, by Theorem \ref{T_1-is-CT1}, it is  $CT_1$. We are done.
\end{proof}

The following examples are more interesting.
\begin{ex}\label{one-point} Let $\alpha D$ be the one-point compactification of an uncountable discrete space $D$. Then,  $\alpha D$
is $CT_1$, $T_2$ and compact but not $CT_2$.
\end{ex}

\begin{proof}
By Theorem \ref{T_1-is-CT1}, we only need to verify that $\alpha D$
is  not $CT_2$.

Let $\alpha D=D\cup\{x\}$. Suppose that $\alpha D$ is $CT_2$, and let $\phi:\mathcal{T}_2(\alpha D)\to \Cld(\alpha D)^2$ be a $CT_2$-map.
Then, for each $d\in D$, there exist two finite sets $A_d$ and $B_d$ such that
\[ \phi(d,x)=(\alpha D\setminus A_d, B_d)\mbox{ and } x\not\in B_d\supseteq A_d\ni d.\]
Since $\phi:\mathcal{T}_2(\alpha D)\to \Cld(\alpha D)^2$ is continuous at $(d,x)$ and $B_d$ is an isolated point in $\Cld(\alpha D)$, there exists a finite set $F_d\subset D$, such that
$\phi_2(d,d')=B_d$ for $d'\in D\setminus F_d$. Thus, $\phi_1(d,d')$ is co-finite for each $d'\in D\setminus F_d$, therefore, $\phi_1(d,d')\ni x$.

Choose a sequence $(d_n)$ consisting of distinct elements from $D$. Then  $F=\bigcup_{n\in\N} F_{d_n}$ is countable.
Hence, there exists $d'\in D\setminus F$.  Then, $\phi_1(d_n,d')\ni x$ for each $n\in\N$. It follows from $d_n\to x$ that  \[ \phi_1(x,d')=\lim_{n\to\infty}\phi_1(d_n,d').\]
Hence, $\phi_1(x,d')\ni x$. This yields a contradiction.
\end{proof}

As usual, we may think of every ordinal as the set of all ordinals less than it.
Moreover,
each ordinal carries the usual order topology. In Example \ref{omega1}, we will show that no uncountable ordinal space is $CT_2$.


\section{Ordinal spaces}\label{ord-sp}

First, we have the following lemma.

 \begin{lem}\label{max-continuous} For each ordinal $\alpha$, the maps  $A\mapsto \sup A$ and $A\mapsto \min A$ are continuous from $\Cld(\alpha)$ to $\alpha+1$.
	\end{lem}
	\begin{proof}
		Let $g:\Cld(\alpha) \to \alpha+1$ be defined by $A\mapsto \sup A$ for any $A\in \Cld(\alpha)$.
		We now prove that the preimages of the two types of elements in the standard subbase of $\alpha+1$ are open sets in $\Cld(\alpha)$.
		
		(1) Let $W = [0,\beta)$ with $\beta \leq \alpha+1$. If $\beta$ is a limit ordinal, then
		\begin{gather*}
		g^{-1}(W)
		= \{ A \in \operatorname{Cld}(\alpha) \mid \sup A < \beta \}\\
		= \bigcup_{ \xi<\beta}\{ A \mid A \subseteq [0,\xi)\}
		= \bigcup_{\xi<\beta}[0,\xi)^+.
\end{gather*}
		And if $\beta=\gamma+1$ is a successor ordinal, then
\[ g^{-1}(W)= \{ A \in \operatorname{Cld}(\alpha) \mid \sup A < \beta \}= \{ A \in \operatorname{Cld}(\alpha) \mid \sup A \leq \gamma \}=[0,\beta)^+.\]
				(2) Let $W = (\beta,\alpha]$ with $\beta<\alpha$. Then
		\[
		g^{-1}(W)
		= \{ A \in \operatorname{Cld}(\alpha) \mid \sup A > \beta \}
		= \{ A \mid A \cap (\beta,\alpha] \neq \varnothing \}
		= (\beta,\alpha]^-.
		\]
		Thus $g$ is continuous.
	
Using the same method, we can show that $A\mapsto \min A$ is also continuous from $\Cld(\alpha)$ to $\alpha$.
\end{proof}

 For an ordinal $\alpha$, let
\[ \mathcal{T}_2^d(\alpha)=\{(x,y)\in \mathcal{T}_2(\alpha)\mid x<y\}.\]
A set $A$ in $\alpha$ is an {\bf upper} (a {\bf lower}) {\bf set} if $x\in A$ and $y>x$ ($y<x$) imply that $y\in A$.
Trivially, a closed set $A$ is an upper (resp. lower) set if and only if there exists $x\in\alpha$ such that $A=[x,\alpha)$
(resp., either $A=\alpha$ or there exists $x\in\alpha$ such that $A=[0,x]$).

\begin{lem}\label{continuity-2}
Let $\alpha$ be an ordinal. Define $I:\mathcal{T}_2^d(\alpha)\cup \triangle(\alpha)\to \Cld(\alpha)$ by
\[ I(x,y)=[x,y].\]
Then, $I$ is continuous. In particular, $x\mapsto [0,x]$ and $x\mapsto [x,\alpha)$ are continuous.
\end{lem}

\begin{proof} Assume that $(x_0,y_0)\in \mathcal{T}_2^d(\alpha)\cup \triangle(\alpha)$.
 We shall prove that for every open set $U$ in $\alpha$, $I^{-1}(U^-)$ and $I^{-1}(U^+)$ are open in $\mathcal{T}_2^d(\alpha)\cup \triangle(\alpha).$

For the former,
without loss of generality, we can assume that $U=(a,b)$, where we think that $[0,b)=(-1,b)$.
If  $I(x_0,y_0)\cap (a,b)\not=\emptyset,$ we consider the following \vspace*{3mm} cases:\\
{\it Case A:} $x_0\in (a,b)$.

Then $((a,b)\times \alpha)\cap (\mathcal{T}_2^d(\alpha)\cup \triangle(\alpha))$ is a neighborhood of $(x_0,y_0)$ and for each element
$(x,y)$ in the neighborhood, we have $I(x,y)\cap (a,b)\ni x$ and \vspace*{3mm} hence  $I(x,y)\cap (a,b)\not=\emptyset$.\\
{\it Case B:} $y_0\in (a,b)$.

The proof is similar to Case A and is therefore\vspace*{3mm} omitted.\\
{\it Case C:} Otherwise.

Then $x_0\leq a<b\leq y_0$ and there exists $z_0\in (a,b)$. Thus $[0,x_0+1)\times (z_0,\alpha)$ is a neighborhood
of $(x_0,y_0)$ and for each $(x,y)\in [0,x_0+1)\times (z_0,\alpha)$, $I(x,y)\cap (a,b)\ni z_0$, thus, $I(x,y)\cap (a,b)\not=\emptyset$. \\

Thus, $I^{-1}((a,b)^-)$ is open in $\mathcal{T}_2^d(\alpha)\cup \triangle(\alpha).$ \\

Now, we verify the latter. Let  $I(x_0,y_0)\subseteq U$.\footnote{For $(\cdot)^-$, it suffices to verify only basic open sets of $\alpha$, since the Vietoris lower set operator preserves arbitrary unions:$	\Big(\bigcup_{i\in I} B_i\Big)^-=\bigcup_{i\in I} B_i^-.$
	However, this union-reserving property fails for the upper set operator $(\cdot)^+$.}
t
If $x_0\neq 0$, then there exists $z_0<x_0$  such that $(z_0,x_0]\subseteq U$.
It follows that $[x_0,y_0]\subseteq (z_0,y_0+1)\subseteq U.$ Consider the neighborhood $((z_0,\alpha)\times [0,y_0+1))\cap ((\mathcal{T}_2^d(\alpha)\cup \triangle(\alpha)))$ of $(x_0,y_0)$. It is trivial to verify that, for each element $(x,y)$ of the neighborhood,
\[ I(x,y)\subseteq (z_0,y_0+1)\subseteq U.\]

If $x_0=0$, then $[x_0,y_0]\subseteq [0,y_0+1)\subseteq U.$ Consider the neighborhood $([0,\alpha)\times [0,y_0+1))\cap ((\mathcal{T}_2^d(\alpha)\cup \triangle(\alpha)))$ of $(x_0,y_0)$. It is trivial to verify that, for each element $(x,y)$ of the neighborhood,
\[ I(x,y)\subseteq [0,y_0+1)\subseteq U.\]

Thus, $I^{-1}(U^+)$ is open in $\mathcal{T}_2^d(\alpha)\cup \triangle(\alpha).$
Therefore, $I:\mathcal{T}_2^d(\alpha)\cup \triangle(\alpha)\to \Cld(\alpha)$ is continuous.

Note that $[0,x]=I(0,x)$ and hence the map $x\mapsto [0,x]$ is continuous.

To show that $f(x)=[x,\alpha)$ is continuous from $\alpha$ to $\Cld(\alpha)$, we note that $[x,\alpha)=\bigcup_{x\leq y<\alpha} I(x,y)$. Thus, for each open set $U$ in $\alpha$, if $f(x)\in U^-$, then there exists
$y\geq x$ such that $I(x,y)\in U^-$, by the continuity of $I$, there exists a neighborhood $V\times W$ of $(x,y)$ in $\alpha^2$ such that $I(V\times W)\subseteq U^-$. It follows that $f(V)\subseteq U^-$. This shows that $f^{-1}(U^-)$ is open in $\alpha$.
Now, assume that $f(x)\in U^+$ for some open set $U$ in $\alpha$. Then there exists $z\in\alpha$ such that $x\in (z,x+1)\subseteq U$. And for each $y\in (z,x+1)$, we have that
\[ f(y)=[y,\alpha)\subseteq (z,x+1)\cup f(x)\subseteq U.\]
That is, $f(z,x+1)\subseteq U^+$. This shows that $f^{-1}(U^+)$ is also open in $\alpha$.  Hence, $x\mapsto [x,\alpha)$ is continuous.
\end{proof}

\begin{lem}\label{ordinal-T_2} For each countable ordinal $\alpha$, there exists a continuous map $f^\alpha:\mathcal{T}_2^d(\alpha)\to \alpha^2$ such that, for each $(x,y)\in \mathcal{T}_2^d(\alpha)$,
\begin{equation}\label{equ-CT2+1} x<f_1^\alpha(x,y);~ f_2^\alpha(x,y)<y~~ \mbox{ and }~~[f_1^\alpha(x,y),\alpha)\cup [0,f_2^\alpha(x,y)]=\alpha.\end{equation}
Therefore, there exists a $CT_2$-map $\phi^\alpha:\mathcal{T}_2^d(\alpha)\to \Cld(\alpha)^2$
such that $\phi_1^\alpha(x,y)$ and $\phi_2^\alpha(x,y)$ are an upper set and a lower set in $\alpha$, respectively.
\end{lem}
\begin{proof}
We show the first statement in Lemma \ref{ordinal-T_2}.
 First, the first statement in Lemma \ref{ordinal-T_2} holds when $\alpha=2$. Indeed, we only need to define $f^2(0,1)=(1,0).$

Now, suppose $\alpha$ is a countable ordinal and the first statement in Lemma \ref{ordinal-T_2} holds for each $\beta<\alpha$.
That is, we have defined a map $f^\beta:\mathcal{T}_2^d(\beta)\to \beta^2$ satisfying our requirements.
We define a map $f^\alpha:\mathcal{T}_2^d(\alpha)\to \alpha^2$ satisfying our requirements that the first statement in Lemma \ref{ordinal-T_2} holds for $\alpha$ by considering the following \vspace*{3mm} cases.\\
{\it Case A:} $\alpha=\beta+2$ for some $\beta$.

Then, Lemma \ref{ordinal-T_2} holds for $\beta+1$. Define $f^\alpha:\mathcal{T}_2^d(\alpha)\to \alpha^2$ as follows:
\[ f^\alpha(x,y)=\left\{
\begin{array}{ll}
f^{\beta+1}(x,y)& y<\beta+1,\\
(\beta+1, \beta) & y=\beta+1.
\end{array}
\right.
\]
Note that $\mathcal{T}_2^d(\alpha)=\mathcal{T}_2^d(\beta+1)\oplus (\beta+1)\times\{\beta+1\}$. It follows that $f^\alpha:\mathcal{T}_2^d(\alpha)\to \alpha^2$ is continuous. Trivially, it also satisfies the other \vspace*{3mm} requirements. \\
{\it Case B:} $\alpha$ is a countable limit ordinal.

Since $\alpha$ is countable, there exists a sequence $\{0=\alpha_0<\alpha_1<\alpha_2<\cdots\}$ such that  $\sup_{n\in\N} \alpha_n=\alpha$
and $\alpha_n=\beta_n+1$ for each $n\in\N$. For each pair $n,m\in\N$ with $n\leq m$, let
\[
\begin{split}
\mathcal{U}_n=\{(x,y)\in [\alpha_{n-1},\alpha_n]^2\mid x<y\}, ~~\\
\mathcal{V}_{n,m}=\{(x,y)\in \alpha^2\mid \alpha_{n-1}\leq x<\alpha_n\leq \alpha_m<y\leq \alpha_{m+1}\}.
\end{split}
\]
Then, $\{\mathcal{U}_n,\mathcal{V}_{n,m}\mid n,m\in\N,n\leq m\}$  is a partition of $\mathcal{T}_2^d(\alpha)$ consisting of open sets in   $\mathcal{T}_2^d(\alpha)$.
Hence, we only need to define  a continuous map $f^\alpha$ on each element of the family.
For each $n\in\N$ and $(x,y)\in\mathcal{U}_n$, let $f^\alpha(x,y)=f^{\alpha_{ n+1}}(x,y).$
Then $f^\alpha:\mathcal{U}_n\to \alpha^2$ satisfies the requirements.
And for $n\le m$ and $(x,y)\in\mathcal{V}_{n,m}$, define the constant  map $f^\alpha(x,y)=(\alpha_n,\alpha_n)$.
Then, it is easy to verify that $f^\alpha:\mathcal{V}_{n,m}\to \alpha^2$ satisfies the \vspace*{3mm} requirements. \\
{\it Case C:} $\alpha=\beta+1$ for some countable limit ordinal $\beta\not=0$.

The proof of Case C is similar to that of Case B. Since $\beta$ is countable, there exists a sequence $\{0=\alpha_0<\alpha_1<\alpha_2<\cdots\}$ such that  $\sup_{n\in\N} \alpha_n=\beta$
and $\alpha_n=\beta_n+1$ for each $n\in\N$. For each pair $n,m\in\N$ with $n\leq m$, let
\begin{gather*}
\mathcal{U}_n=\{(x,y)\in [\alpha_{n-1},\alpha_n]^2\mid x<y\},\\
\mathcal{U}=\bigcup_{n\in\N}\mathcal{U}_n,\\
\mathcal{V}_{n,m}=\{(x,y)\in \alpha^2\mid \alpha_{n-1}\leq x<\alpha_n\leq \alpha_m<y\leq \alpha_{m+1}\},\\
\mathcal{C}_n=\{(x,\beta)\mid x\in [\alpha_{n-1},\alpha_n)\},\\
\mathcal{C}=\bigcup_{n\in\N}\mathcal{C}_n.
\end{gather*}
Then, $\{\mathcal{U}_n,\mathcal{V}_{n,m}, \mathcal{C}_n\mid n,m\in\N,n\leq m\}$ is a partition of $\mathcal{T}_2^d(\alpha)$. Moreover, $\mathcal{U}_n$ is clopen set in the clopen set $\mathcal{U}$ in $\mathcal{T}_2^d(\alpha)$ for $n\in\N$, $ \mathcal{V}_{n,m}$ is open for $ n,m\in\N$ with $n\leq m$,  $\mathcal{C}_n$
is clopen in the closed subspace $\mathcal{C}$ of $\mathcal{T}_2^d(\alpha)$. First, we define  $f^\alpha$ for each element of the family $\{\mathcal{U}_n,\mathcal{V}_{n,m}, \mathcal{C}_n\mid n,m\in\N,n\leq m\}$ to $ \alpha^2$ as follows: For $\mathcal{U}_n$ and $\mathcal{V}_{n,m}$, we use the same formulas as in Case B. And for each $n\in\N$, define the constant  map $f^\alpha: \mathcal{C}_n\to \alpha^2$ by
$f^\alpha(x,\beta)=(\alpha_n,\alpha_n).$
Trivially, $f^\alpha: \mathcal{C}_n\to \alpha^2$ satisfies the requirements. It follows that  $f^\alpha: \mathcal{C}\to \alpha^2$ satisfies the requirements.
Moreover, $\mathcal{T}_2^d(\alpha)$ is metrizable. Therefore, to show that $f^\alpha: \mathcal{T}_2^d(\alpha)\to \alpha^2$ satisfies the requirements for $\alpha$,
it suffices to verify that for each sequence $(x_i,y_i)\in \mathcal{V}_{n_i,m_i}$ and $(x,\beta)\in \mathcal{C}_n$, if $(x_i,y_i)\to (x,\beta)$, then
$f^\alpha(x_i,y_i)\to f^\alpha(x,\beta).$
Indeed, since $(x_i,y_i)\to (x,\beta)$, we have that $n_i=n $ and $m_i>n$ for large enough $i$. It follows from the definition of $f^\alpha$ that $f^\alpha(x_i,y_i)=(\alpha_n,\alpha_n)=f^\alpha(x,\beta)$ for large enough $i$.

Now, we verify the second statement. Let
\[ \phi^{\alpha}(x,y)=\left([f^\alpha_1(x,y),\alpha), [0,f^\alpha_2(x,y)]\right),~~(x,y)\in\mathcal{T}_2^d(\alpha).\]
It follows from the first statement and Lemma \ref{continuity-2} that $\phi^{\alpha}:\mathcal{T}_2^d(\alpha)\to \Cld(\alpha)^2$ is continuous.
By Equation (\ref{equ-CT2+1}), $\phi^{\alpha}$ satisfies also the other conditions in Lemma \ref{ordinal-T_2}.
 \end{proof}

\begin{lem}\label{partition}
Let $\alpha$ be an ordinal and
 \begin{gather*}
\mathcal{L}=\{(x,B)\in\mathcal{T}_3(\alpha)\mid x<b\mbox{ for each }b\in B\};\\
\mathcal{R}=\{(x,B)\in\mathcal{T}_3(\alpha)\mid x>b\mbox{ for each }b\in B\};\\
\mathcal{M}=\mathcal{T}_3(\alpha)\setminus (\mathcal{R}\cup\mathcal{L}).
\end{gather*}
Then, $\mathcal{T}_3(\alpha)=\mathcal{L}\oplus \mathcal{R}\oplus\mathcal{M}$.
\end{lem}

\begin{proof}
Trivially, $\mathcal{L},~\mathcal{R} $ and $\mathcal{M}$ are a partition of the set $\mathcal{T}_3(\alpha)$.
Hence, we only need to verify that each of them is open in $\mathcal{T}_3(\alpha)$.
For each $(x,B)\in \mathcal{L}$, $[0,x+1)\times [x+1,\alpha)^+$
is a neighborhood of $(x,B)$ in $\mathcal{T}_3(\alpha)$ and  misses $\mathcal{R}\cup \mathcal{M}$. Hence, $\mathcal{L}$ is open in $\mathcal{T}_3(\alpha)$.
For each $(x,B)\in \mathcal{R}$, $(\max B,\alpha)\times [0,\max B+1)^+$
is a neighborhood of $(x,B)$ in $\mathcal{T}_3(\alpha)$ and  misses $\mathcal{L}\cup \mathcal{M}$. Hence, $\mathcal{R}$ is open in $\mathcal{T}_3(\alpha)$.
 Moreover, for each $(x,B)\in \mathcal{M}$,
 \[ (\max (B\cap [0,x]),x+1)\times \left([0,\max (B\cap [0,x])+1)^-\cap [x+1,\alpha)^-\right)\]
   is a neighborhood of $(x,B)$ in $\mathcal{T}_3(\alpha)$ and misses $\mathcal{L}\cup \mathcal{R}$.
This shows that $\mathcal{M}$ is also open in $\mathcal{T}_3(\alpha)$.
\end{proof}

\begin{lem}\label{continuity-1}
Let $\alpha$ be an ordinal. Define $l:\mathcal{M}\cup \mathcal{L}\to \mathcal{L}$ and $r:\mathcal{M}\cup \mathcal{R}\to \mathcal{R}$  as
\[ l(x,B)=(x,[x,\alpha)\cap B),~~r(x,B)=(x,[0,x]\cap B).\]
Then, both $l$ and $r$ are continuous.
\end{lem}

\begin{proof}We only need to verify $l_2(x,B)=[x,\alpha)\cap B$ and $r_2(x,B)=[0,x]\cap B$ are continuous from $\mathcal{M}$ to $\Cld(\alpha)$.
For each $(x_0,B_0)\in \mathcal{M}$, let $m=\max (B_0\cap [0,x_0])$ and $M=\min (B_0\cap [x_0,\alpha))$. Then $m<x_0<x_0+1\leq M$.

For each open set $(a,b)$ in $\alpha$ with $l_2(x_0,B_0)\in (a,b)^-$, we have that
\[ [x_0,\alpha)\cap B_0\cap (a,b)\not=\emptyset.\]
Since $x_0\not\in B_0$, we have that $B_0\in\left( (x_0, \alpha)\cap (a,b)\right)^-$. Thus, $[0,x_0+1)\times \left( (x_0, \alpha)\cap (a,b)\right)^-$
is a neighborhood of $(x_0,B_0)$ and, for each $(x,B)\in [0,x_0+1)\times \left( (x_0, \alpha)\cap (a,b)\right)^-$, we have that $l_2(x,B)\in (a,b)^-$.

For each open set $U$ in $\alpha$ with $l_2(x_0,B_0)\in U^+$, we have that
\[ [x_0,\alpha)\cap B_0\subseteq U.\]
It follows that $B_0\subset U\cup [0,m+1)$. Thus, $(x_0,B_0)\in (m,x_0+1)\times \left(U\cup [0,m+1)\right)^+$. Moreover, for each $(x,B)\in (m,x_0+1)\times \left(U\cup [0,m+1)\right)^+$, we have that
\[ l_2(x,B)=[x,\alpha)\cap B\subseteq (m,\alpha)\cap \left(U\cup [0,m+1)\right)= (m,\alpha)\cap U\subseteq U.\]
That is, $l_2(x,B)\in U^+$.

We have shown that $l_2:\mathcal{M}\to\Cld(\alpha)$ is continuous.

Using a similar argument, we can show that  $r_2:\mathcal{M}\to\Cld(\alpha)$ is also continuous.
\end{proof}

\begin{rem} From the proof of Lemma \ref{continuity-1} it follows that $l_2(x,B)=[x,\alpha)\cap B$
is continuous from $\mathcal{M}\cup \mathcal{L}$ to $\Cld(\alpha)$. However, the extension of $l_2$ to $\{(x,B)\in \alpha\times \Cld(\alpha)\mid l_2(x,B)\not=\emptyset\}$
is not continuous. Indeed, let $\alpha=\omega_0+2$ and $B_n=\{n,\omega_0+1\}$. Then $B_n\to \{\omega_0,\omega_0+1\}$.
But,
\[ l_2(\omega_0,B_n)=\{\omega_0+1\}\to \{\omega_0+1\}\not=\{\omega_0,\omega_0+1\}= l_2(\omega_0,\{\omega_0,\omega_0+1\}).\]
Thus, for Lemma \ref{continuity-2}, although $I(x,y)=r_2(y,l_2(x,\alpha)),$ we
 cannot obtain a simple proof of Lemma \ref{continuity-2} by Lemma \ref{continuity-1}.
 \end{rem}

\begin{ex}\label{countable-ordinal} Each countable ordinal  space  is $CT_3$.
\end{ex}
\begin{proof} Let $\alpha$ be a countable ordinal. By Lemma \ref{partition}, we only need
to define three $\mathcal{T}_3$ maps $\phi^\mathcal{L}:\mathcal{L}\to \Cld(\alpha)^2$,~$\phi^\mathcal{R}:\mathcal{R}\to \Cld(\alpha)^2$ and $\phi^\mathcal{M}:\mathcal{M}\to \Cld(\alpha)^2$.

 Define $\phi^\mathcal{L}:\mathcal{L}\to \Cld(\alpha)^2$ as follows:  For each $(x,B)\in \mathcal{L}$,
\begin{gather*}
\phi^\mathcal{L}(x,B)=\phi^\alpha(x,\min B).
\end{gather*}
Here, $\phi^\alpha:\mathcal{T}_2(\alpha)\to \Cld(\alpha)^2$ satisfies the conditions in Lemma \ref{ordinal-T_2}. Using Lemmas \ref{max-continuous} and \ref{ordinal-T_2}, it is not hard to verify that $\phi^\mathcal{L}$ is a $CT_3$-map.

Using the same method, we can define a $CT_3$-map $\phi^\mathcal{R}:\mathcal{R}\to \Cld(\alpha)^2$.

Now, for $(x,B)\in\mathcal{M}$, by Lemma \ref{continuity-1}, we can define
 \begin{gather*}
\phi^\mathcal{M}(x,B)=\left(\phi_1^\mathcal{L}\left(l(x, B)\right)\cup \phi_1^\mathcal{R}\left(r(x, B)\right),\phi_2^\mathcal{L}\left(l(x, B)\right)\cap \phi_2^\mathcal{R}\left(r(x,B)\right)\right).
\end{gather*}
By Lemmas \ref{cap-cup} and \ref{continuity-1}, $\phi_1^\mathcal{M}:\mathcal{M}\to\Cld(\alpha)$ is continuous. Note that for $(x,B)\in \mathcal{M}$,
\[ \phi_2^\mathcal{M}(x,B)=I\left(\min \phi_2^\mathcal{R}\left(r(x, B)\right),\max \phi_2^\mathcal{L}\left(l(x, B)\right)\right). \]
Using Lemmas \ref{max-continuous}, \ref{continuity-1} and \ref{continuity-2},
this shows that $\phi_2^\mathcal{M}:\mathcal{M}\to\Cld(\alpha)$ is also continuous. It is straightforward to verify that the map $\phi^\mathcal{M}$ also satisfies the other conditions.
Hence, $\alpha$ is $CT_3$.
\end{proof}

From Theorem \ref{0-dim-CT4} it follows that the following examples hold:

\begin{ex}\label{successive-ordinal-T4} Each countable successor ordinal space is $CT_4$.
\end{ex}

However, we are going to show that no uncountable ordinal space is  $CT_2$. As usual, $\omega_1$ is the first uncountable ordinal and hence, it is the set of all countable ordinals.
An ordinal $\alpha$ is called {\bf regular} if $\cf (\alpha)=\alpha$. For example, $\omega_1$ is regular. A set $S$ in $\alpha$ is called a  {\bf stationary set} if it intersects each unbounded closed set  in $\alpha$. First, we give the following lemma.

\begin{lem}{\rm (\cite{Fodor-1956}, cf. \cite[Theorem 4.41]{Levy-2005} Fodor Pressing-Down Lemma)} \label{Fodor} Let $\alpha$ be a regular ordinal.
For each map $r:\alpha\setminus\{0\}\to \alpha$ with $r(\xi)<\xi$ for each $\xi\in\alpha \setminus\{0\}$, there exist a stationary set $S$ in $\alpha$ and $\gamma\in\alpha$ such that $r(s)=\gamma$ for each $s\in S$.
\end{lem}

\begin{ex} \label{omega1} No uncountable ordinal space is  $CT_2$.
\end{ex}

\begin{proof}Let $\alpha$ be an   uncountable ordinal. If $\alpha$ is $CT_2$, then there exists a $CT_2$-map $\phi:\mathcal{T}^d_2(\alpha)\to \Cld(\alpha)^2$.
In particular, for each countable limit ordinal $\lambda$, $\phi_1(\lambda,\lambda+1)\not\ni \lambda$, that is, $\phi_1(\lambda,\lambda+1)\in (\alpha\setminus\{\lambda\})^+$.
It follows that there exists $\mu_\lambda<\lambda$ such that
$\phi_1(\mu,\lambda+1)\in (\alpha\setminus\{\lambda\})^+$, i.e., $\phi_1(\mu,\lambda+1)\not\ni \lambda$, for $\mu\in (\mu_\lambda,\lambda+1)$.
Thus,
\begin{equation}\label{equ-200} \phi_2(\mu,\lambda+1)\ni \lambda ~~\mbox{for}~~\mu\in (\mu_\lambda,\lambda+1).\end{equation}
Now, we define a map $r:\omega_1\setminus\{0\}\to \omega_1$ as follows

\[
r(\xi)=\left\{
\begin{array}{ll}
\mu_\xi & \mbox{if} ~\xi ~\mbox{is limit},\\
\xi-1 &  \mbox{if} ~\xi ~\mbox{is a successor ordinal}.
\end{array}
\right.
\]
\nolinenumbers

Then $r(\xi)<\xi$ for each $\xi\in\omega_1\setminus\{0\}$. Using Lemma \ref{Fodor}, there exist a stationary set $S$ in $\omega_1$ and $\gamma\in\omega_1$ such that $r(s)=\gamma$ for each $s\in S$. Choose a countable ordinal $\eta>\gamma$ and a strictly increasing sequence $(\mu_n)_{n<\omega}$ in $S$ with each $\mu_n$ a limit ordinal and $\mu_n>\eta$ for each $n$, and let $\mu=\sup_n \mu_n\in\omega_1$. By Equation (\ref{equ-200}), we have that
\begin{equation*}\label{equ-100} \phi_2(\eta,\mu_n+1)\ni \mu_n.\end{equation*}
Note that $\mu_n\to \mu$ and $\mu_n+1\to \mu$. By the continuity of $\phi_2$ in the second variable at the point $(\eta,\mu)$, we have
\[
\phi_2(\eta,\mu)=\lim_{n\to\infty} \phi_2(\eta,\mu_n+1)\ni\mu,
\]
a contradiction.
\end{proof}

\section{Subspaces of real numbers}\label{subspaces-of-R}

For a subspace $X$ of $\R$, let
\[ \mathcal{T}_2^d(X) =\{(x,y)\in X^2\mid x<y\}.\]
If there exists a continuous map $f:\mathcal{T}_2^d(X) \to X$ such that $x_1<f(x_1,x_2)<x_2$ for each pair $(x_1,x_2)\in\mathcal{T}_2^d(X) $, then we say that $X$ has the {\bf continuous interval property} ({\bf CIP}, in short) and $f$ is a {\bf continuous interval map}.
We have the following general result:

\begin{thm}\label{rations-general} Let $X$ be a subspace of $\R$ with the CIP. Then $X$ is  $CT_2$.	
\end{thm}

\begin{proof} Let $f$ be a continuous interval map for $X$.
		For $(x,y)\in\mathcal{T}_2^d(X) $, let	
	\[ \phi(x,y)=\left([f(x,y),+\infty)\cap X,(-\infty,f(x,y)]\cap X\right).\] 	
	Trivially, $\phi_1(x,y)$ and $\phi_2(x,y)$ are in $\Cld(X)$, and $x\not\in \phi_1(x,y)$, $y\not\in \phi_2(x,y)$	
	and $\phi_1(x,y)\cup \phi_2(x,y)=X$. Hence, we only need to verify that $\phi$ is continuous. To this end, since $f$ is continuous, it suffices, for example, to show that $\psi:X\to \Cld(X)$, where $\psi(x)=(-\infty, x]\cap X$, is continuous.	
	
	For each open set $U$ in $X$, if $U\cap \psi(x)\not=\emptyset$, then $U\cap (-\infty,x) \not=\emptyset$ or $x\in U$. If the former holds, it is trivial to verify that there exists $\delta>0$ such that $U\cap (-\infty,x') \not=\emptyset$ when $x'\in X$ with $|x'-x|<\delta$. If the latter holds, then $U$ is an open neighborhood of $x$ in $X$ and $U\cap \psi(x')\not=\emptyset$ for each $x'\in U$.
		
	Moreover, suppose that $U$ is open in $X$ and $\psi(x)\subseteq U$. Then there exists $\delta>0$ such that $X\cap (x-\delta,x+\delta)\subseteq U.$ It follows that	
	$V=X\cap (x-\delta,x+\delta)$ is an open neighborhood of $x$ and $\psi(x')\subseteq U$ for $x'\in V$.	
\end{proof}

Using the above theorem, we immediately obtain the following family of examples of $CT_2$ spaces:

\begin{ex} All subfields of $\R$ are $CT_2$. In particular, the space of rational numbers $\Q$ is $CT_2$.
 \end{ex}

  Moreover, we have the following family of examples of $CT_2$ spaces:

\begin{ex} Let $F$ be a subfield of $\R$ with $F\not=\R$. Then $\R\setminus F$ is $CT_2$. In particular, the space of irrational numbers $\PP$ is $CT_2$.
\end{ex}

\begin{proof} Since $0\in F$, the space $\R\setminus F=((-\infty,0)\setminus F)\oplus ((0,+\infty)\setminus F)$. These two subspaces are homeomorphic. By Theorem \ref{sum}, it therefore suffices to prove that $\R^+\setminus F$ is $CT_2$.

We only need to verify that $X=\R^+\setminus F$ has the CIP.  Since $\Q\subseteq F$, we can write each element in $X$ as a non-terminating binary expansion. That is, we can write each element $x$ in $X$ as $(\cdots,0,x_{n},x_{n-1},\cdots,x_0,x_{-1},x_{-2},\cdots)$, where $x_i\in\{0,1\}$ for $i\in \{n,n-1,\cdots,0\}\cup (-\N)$. For every pair $(x,y)\in \mathcal{T}_2^d(X) $, let
\[ m(x,y)=\max\{i\mid x_i\not=y_i\},~~n(x,y)=\max\{i<m(x,y)\mid x_i=0\}.\]
Since the sets $\{i>0\mid x_i=1\}$ and $\{i>0\mid y_i=1\}$ are finite and $x<y$, $m(x,y)$ is well-defined; and since $x\not\in\Q$, $n(x,y)$ is well-defined.
Now, let $f(x,y)=x+2^{n(x,y)}.$ That is,	
	\[ f(x,y)_i=\left\{	
	\begin{array}{ll}		
		x_i & i\not=n(x,y),\\		
		1 & i=n(x,y).		
	\end{array}
		\right.\]

Since $F$ is a field and $f(x,y)-x=2^{n(x,y)}\in\Q\subseteq F$, we have that $f(x,y)\in X$. Moreover, $x_{m(x,y)}=x_{n(x,y)}=0, y_{m(x,y)}=1$ and $m(x,y)>n(x,y)$, $x<f(x,y)<y$.

Note that in $X$, $z^{(k)}\to z$ if and only if for each $n\in -\N$, there exists $K\in\N$ such that $z^{(k)}_i=z_i$ for every $k>K$
and $i>n$. Using this fact, it is easy to verify that $f:\mathcal{T}_2^d(X) \to X$ is continuous.
\end{proof}

For a 0-dimensional closed subspace $X$ of $(-1,1)$, there exists a countable pairwise disjoint family $\{(a_n,b_n)\mid n\in\N\}$ such that
\[ (-1,1)\sm X=\bigcup_{n\in\N} (a_n,b_n).\]
Moreover, we can require that $b_n-a_n\geq b_m-a_m$ if $n<m$.

\begin{lem}\label{n-map} Let $X$ be a 0-dimensional closed subspace of $(-1,1)$. For each $(x,y)\in\mathcal{T}^d_2(X)$, let $n(x,y)$ be the smallest $n$ such that $x\leq a_n<b_n\leq y$. Then the map $(x,y)\mapsto n(x,y)$ is locally constant.
\end{lem}

\begin{proof}First, since $X$ is 0-dimensional, for each $(x,y)\in \mathcal{T}^d_2(X)$, $n(x,y)$ exists. Second, it is not hard to verify that for each pair $(x',y')\in \mathcal{T}^d_2(X)$, if $|x'-x|<b_n-a_n$ and $|y'-y|<b_n-a_n$,
then $n(x',y')=n(x,y)$.
\end{proof}

Using this lemma, we can show a version of Lemma \ref{ordinal-T_2} for a 0-dimensional closed subspace $X$ of $(-1,1)$ as follows:

\begin{lem}\label{lem-0-dim-closed} For each 0-dimensional closed subspace $X$ of $(-1,1)$, there exists a continuous map $f:\mathcal{T}_2^d(X)\to X^2$ such that, for each $(x,y)\in \mathcal{T}_2^d(X)$,
\[ x<f_1(x,y);~ f_2(x,y)<y~~ \mbox{ and }~~([f_1(x,y),1)\cap X)\cup ((-1,f_2(x,y)]\cap X)=X.\]
Therefore, there exists a $CT_2$-map $\phi:\mathcal{T}_2^d(X)\to \Cld(X)^2$
such that $\phi_1(x,y)$ and $\phi_2(x,y)$ are an upper set and a lower set in $X$, respectively.
\end{lem}

\begin{proof} For $(x,y)\in \mathcal{T}_2^d(X)$, let
\[ f(x,y)=(b_{n(x,y)},a_{n(x,y)}).\]
It follows from Lemma \ref{n-map} that $f:\mathcal{T}_2^d(X)\to X^2$ is continuous. Trivially, it also satisfies the other conditions.
\end{proof}

\begin{ex} Let $X$ be a 0-dimensional closed subspace  of $(-1,1)$. Then $X$ is $CT_3$, and it is $CT_4$ if and only if it is compact.
\end{ex}

\begin{proof} Replacing Lemma \ref{ordinal-T_2}  by Lemma \ref{lem-0-dim-closed}, the construction used in the proof of Example \ref{countable-ordinal} carries over verbatim (with the corresponding order intervals in X). Thus, every 0-dimensional closed subspace $X$ of $(-1,1)$ is $CT_3$. It follows from
Theorem \ref{0-dim-CT4} and Corollary \ref{nocomact-no-CT_4} that $X$ is $CT_4$ if and only if it is compact.
\end{proof}

\section{Summary and questions}
In this paper, we define  $CT_i$-spaces for $i=1,2,3,4$. Let $CT_i$ be the class of all  $CT_i$-spaces and let $T_i$ be the class of all $T_i$-spaces. The following relations were proved:
\[
\begin{array}{llll}
T_1 & = & CT_1&\supset  \mbox{ uncountable ordinal spaces,}~ \alpha D \mbox{ for uncountable discrete spaces }D\\
\cup & & \cup &\\
 T_2 & \supset  & CT_2&\ni \Q,\PP, \N\cup\{p\},S_\omega,\R_l  \\
            \cup && \cup\\
T_3 & \supset  & CT_3&\ni \R^n, [0,+\infty),  \mbox{ countable limit ordinal }\alpha, \\
                \cup &         &\cup &        \\
 T_4 & \supset  & CT_4 & \ni \II^m, \sphere^n,  C, \mbox{ countable successor ordinal }\alpha. \\
\end{array}
\]
where $n\in\N$ and $m\in\N\cup\{\infty\}$. However, $CT_4\not\ni \R^n,\Q, \PP$, $CT_3\not\ni \N\cup\{p\},S_\omega$ and
$CT_2\not\ni \alpha$ for each uncountable ordinal $\alpha$. We leave the following questions open:

\begin{prob} Is every dense subspace of $\R$ $CT_3$
?
\end{prob}

\begin{prob} Are $\Q$ and $\PP$ $CT_3$?
\end{prob}

Note that all $CT_3$ spaces given in this paper are metrizable. Hence, we pose the following problem:

 \begin{prob}\label{n-metr-CT3} Is there a non-metrizable $CT_3$-space, or even, a non-metrizable $CT_4$-space?
\end{prob}

In particular, we conjecture that the following problem has an affirmative answer:

 \begin{prob}\label{metr-CT3} For a space $X$ with a unique non-isolated point, if $X$ is $CT_3$, must $X$ be metrizable?
\end{prob}

Moreover, we can consider the following problems:

\begin{prob} Is every compact $CT_2$-space $CT_3$, or even $CT_4$?
\end{prob}

\begin{prob} Can one establish stronger versions of the Urysohn Lemma and the Tietze Extension Theorem for $CT_4$-spaces?
\end{prob}

\section*{Declaration of competing interest}

The authors declare that they have no known competing financial interests or personal relationships that could have
appeared to influence the work reported in this paper.

\section*{Acknowledgement}
The authors would like to express their appreciation to Professor Shuguo Zhang of
Sichuan University for his  valuable suggestions.

\end{document}